\documentclass[11pt, reqno]{amsart}
\usepackage{amsmath, amsthm, amscd, amsfonts, amssymb, graphicx, color}
\usepackage[bookmarksnumbered, colorlinks, plainpages]{hyperref}
\hypersetup{colorlinks=true,linkcolor=red, anchorcolor=green, citecolor=cyan, urlcolor=red, filecolor=magenta, pdftoolbar=true}

\usepackage[a4paper]{geometry}

\newtheorem{theorem}{Theorem}[section]
\newtheorem*{theorem*}{Theorem}
\newtheorem{lemma}[theorem]{Lemma}
\newtheorem{proposition}[theorem]{Proposition}
\newtheorem{corollary}[theorem]{Corollary}
\theoremstyle{definition}
\newtheorem{definition}[theorem]{Definition}
\newtheorem{example}[theorem]{Example}

\theoremstyle{remark}
\newtheorem{remark}[theorem]{Remark}
\numberwithin{equation}{section}

\usepackage{stmaryrd,dsfont,bm}
\usepackage[T1]{fontenc}
\usepackage{hyperref}
\usepackage[matrix,arrow,curve]{xy}
\CompileMatrices

\newcommand{\C}{\mathds{C}}

\DeclareMathOperator{\id}{id}

\DeclareMathOperator{\Irr}{Irr}

\DeclareMathOperator{\Dil}{\mathfrak{Dil}}
\newcommand{\ev}{\mathrm{ev}}

\newcommand{\qG}{\mathds{G}}
\newcommand{\qB}{\mathds{B}}
\newcommand{\dqG}{\hat{\mathds{G}}}

\newcommand{\Glob}{\mathfrak{G}}

\newcommand{\new}{}
\newcommand{\old}{}

\begin{document}

 \title[Partial actions   of ${C^{*}}$-quantum groups]{Partial actions  of $\bm{C^{*}}$-quantum groups}

\author[F. Kraken, P. Quast, and T. Timmermann]{Franziska Kraken$^{1}$, Paula Quast$^{2}$, and Thomas Timmermann$^{3}$$^{*}$} 

 \email{$^{1}$\textcolor[rgb]{0.00,0.00,0.84}{franzi.kraken@googlemail.com}}
\email{$^{2}$\textcolor[rgb]{0.00,0.00,0.84}{paula.quast@web.de}}

\address{$^{3}$FB Mathematik und Informatik, University of Muenster \\ Einsteinstr.\ 62, 48149
  Muenster, Germany}
 \email{\textcolor[rgb]{0.00,0.00,0.84}{timmermt@math.uni-muenster.de}}


 
\date{\today
\newline \indent $^{*}$Corresponding author}

  \thanks{Partially supported by the SFB 878
      ``Groups, geometry and actions'' funded by the DFG and by the grant H2020-MSCA-RISE-2015-691246-QUANTUM DYNAMICS}
 
\begin{abstract}
  Partial actions of groups on $C^{*}$-algebras and the closely related actions and coactions of Hopf algebras received much attention over the last decades. They arise naturally as restrictions of their global counterparts to non-invariant subalgebras, and the ambient eveloping global (co)actions have proven useful for the study of associated crossed products. In this article, we introduce the partial coactions of $C^{*}$-bialgebras, focussing on $C^{*}$-quantum groups, and prove existence of an enveloping global coaction under mild technical assumptions.  We also show that partial coactions of the function algebra of a discrete group correspond to partial actions on direct summands of a $C^{*}$-algebra, and relate partial coactions of a compact or its dual discrete $C^{*}$-quantum group to partial coactions or partial actions of the dense Hopf subalgebra. \new As a fundamental example, we associate to every  discrete $C^{*}$-quantum group a quantum Bernoulli shift.\old
\end{abstract}

\maketitle

\section{Introduction}

Partial actions of groups on spaces and on $C^{*}$-algebras were gradually
introduced in \cite{exel:circle}, \cite{exel:twisted}, \cite{mcclanahan}, with
more recent study of associated crossed products  shedding new light on the inner structure of many interesting $C^{*}$-algebras; see \cite{exel:book} for a comprehensive introduction and an overview.  In the purely algebraic setting, the corresponding notion of a partial action or a partial coaction of a Hopf algebra on an algebra was introduced in \cite{caenepeel}.

Naturally, such partial (co)actions arise by restricting global (co)actions to
non-invariant subspaces or ideals, and in these cases, all the tools that are
available for the study of global situation can be applied to the study of the
partial one. Therefore, it is highly desirable to know, given a partial group
action or a partial Hopf algebra (co)action, whether it can be identified with
some restriction of a global one, whether there exists a minimal global one ---
called a \emph{globalization} --- and whether the latter, if it exists, can be
constructed explicitly. For partial actions of groups on locally compact
Hausdorff spaces, such a globalization can always be constructed, but the
underlying space need no longer be Hausdorff \cite{abadie:phd},
\cite{abadie:takai}. As a consequence, partial actions of groups on
$C^{*}$-algebras can not always be identified with the restriction of a global
action \cite{abadie:takai}. In the purely algebraic setting, partial (co)actions
of Hopf algebras always have a globalization \cite{alves:enveloping},
\cite{alves:globalization-co}; see also
\cite{alvares:partial-module-categories}, \cite{alves:globalization-twisted},
\cite{castro:weak-partial}.

In this article, we introduce partial coactions of $C^{*}$-bialgebras, in particular, of $C^{*}$-quantum groups, on $C^{*}$-algebras, and relate them to the partial (co)actions discussed above.  In case of the function algebra of a discrete group, partial coactions correspond to partial actions of groups where for every group element, the associated domain of definition is a direct summand of the total $C^{*}$-algebra, and these are precisely the partial actions for which existence of a globalization can be proven. If the $C^{*}$-bialgebra is a discrete $C^{*}$-quantum group, then every partial coaction gives rise to a partial action of the Hopf algebra of matrix coefficients of the dual compact quantum group.  Finally, in case of a compact $C^{*}$-quantum group, partial coactions restrict, under a natural condition, to partial coactions of the Hopf algebra of matrix coefficients on a dense subalgebra.

\new

Partial coactions appear naturally as restrictions of ordinary coactions to ideals or, more generally, to  $C^{*}$-subalgebras that are weakly invariant in a suitable sense. An identification of a partial coaction with such a restriction will be called a dilation of the partial coaction.  
The main result of this article is the existence and a construction of a  minimal dilation, also called a globalization,  under mild assumptions.  We follow the approach for coactions of Hopf algebras \cite{alves:globalization-co}, but face new technical difficulties. To deal with these, we assume that the $C^{*}$-algebra of the quantum group under consideration has the slice map property, which follows, for example, from nuclearity \cite{wassermann:slice}, and is automatic if the quantum group is discrete.  
Briefly, the main result can be summarised as follows.
\begin{theorem*}
  Let $(A,\Delta)$ be a $C^{*}$-quantum group, where $A$ has the slice map property. Then every injective, weakly continuous, regular partial coaction  of $(A,\Delta)$ has a minimal dilation and the latter is unique up tio isomorphism.
\end{theorem*}
   Presently, we do not see whether this slice map assumption is just convenient or genuinely necessary.  

\old

Parts of the results in this article were obtained in the Master's theses of the first and the second author.  In following articles, we plan to study crossed products for partial coactions, and partial corepresentations of $C^{*}$-bialgebras.

The article is organized as follows.  In Section \ref{sec:prelim}, we recall background on $C^{*}$-quantum groups, strict $*$-homo\-morphisms and the slice map property.  In Section \ref{sec:partial}, we introduce partial coactions of $C^{*}$-bialgebras and discuss a few desirable properties like weak and strong continuity.  In Section \ref{sec:groups}, we show that partial actions of a discrete group $\Gamma$ on a $C^{*}$-algebra correspond to counital partial coactions of the function algebra $C_{0}(\Gamma)$ if and only if the domains of definition are direct summands of the $C^{*}$-algebra.  In Section \ref{sec:hopf}, we relate partial coactions of compact and of discrete $C^{*}$-quantum groups to coactions and actions of the Hopf algebra of matrix elements of the compact quantum group.  In Section \ref{sec:restriction-morphisms}, we show how partial coactions arise from global ones by restriction, and discuss the closely related notion of weak or strong morphisms between partial coactions. \new In Section \ref{section:bernoulli}, we construct for every discrete quantum group a quantum a quantum Bernoulli shift and obtain, by restriction, a partial coaction that is initial in a suitable sense. \old In Section \ref{sec:dilation}, we consider the situation where a partial coaction can be identified with the restriction of a global coaction, and study a few preliminary properties of such identifications.  Finally, in Section \ref{sec:minimal}, we prove the main result stated above.

\section{Preliminaries} \label{sec:prelim}

Let us fix some notation and recall some background.
\subsection*{Conventions and notation} 
Given a locally compact Hausdorff space $X$, we denote by $C_{b}(X)$ and $C_{0}(X)$ the $C^{*}$-algebra of continuous functions that are bounded or vanish at infinity, respectively.

For a subset $F$ of a normed space $E$, we denote by $[F] \subseteq E$ its closed linear span. 

Given a $C^{*}$-algebra $A$, we denote by $A^{*}$ the space of bounded linear functionals on $A$, by $M(A)$ the multiplier algebra and by $1_{A} \in M(A)$ the unit of $M(A)$.  

Given a  Hilbert space $K$, we denote by $1_{K}$ the identity on $H$.

Let $A$ and $B$ be $C^{*}$-algebras. A $*$-homomorphism $\varphi\colon A\to M(B)$ is called \emph{nondegenerate} if $[\varphi(A) B]=B$. Each nondegenerate $*$-homomorphism $\varphi\colon A\to M(B)$ extends uniquely to a unital $*$-homomorphism from $M(A)$ to $M(B)$, which we denote by $\phi$ again. By a \emph{representation} of a $C^{*}$-algebra $A$ on a Hilbert space $H$ we mean a $*$-homomorphism  $\pi \colon A \to \mathcal{B}(H)$. All tensor products of $C^{*}$-algebras will be minimal ones. 

We write $\sigma$ for the tensor flip isomorphism $A\otimes B\to B\otimes A$,  $a \otimes b \mapsto b\otimes a$.

\subsection*{$C^{*}$-bialgebras and $C^{*}$-quantum groups} 

A \emph{$C^{*}$-bialgebra} is a $C^{*}$-algebra $A$ with a non-degenerate $*$-homomorphism $\Delta\colon A \to M(A\otimes A)$, called the \emph{comultiplication}, that is coassociative in the sense that $(\Delta \otimes \id_{A})\circ \Delta = (\id_{A} \otimes \Delta) \circ \Delta$. It satisfies the \emph{cancellation conditions} if
\begin{equation}
  \label{eq:Cancellations}
  [\Delta(A)(1_A\otimes A)] = A\otimes A       = [(A\otimes 1_A)\Delta(A)].
\end{equation}

Given a $C^{*}$-bialgebra $(A,\Delta)$, the dual space $A^{*}$ is an algebra with respect to the convolution product defined by $\upsilon \omega:=(\upsilon \otimes \omega)\circ \Delta$.

A \emph{counit} for a $C^{*}$-bialgebra $(A,\Delta)$ is a character $\varepsilon$  on $A$ satisfying $(\varepsilon \otimes \id_{A})\circ \Delta = \id_{A} = (\id_{A} \otimes \varepsilon)\circ \Delta$. If it exists, such a counit is a unit in the algebra $A^{*}$ and thus unique.

A \emph{morphism} of $C^{*}$-bialgebras $(A,\Delta_{A})$ and $(B,\Delta_{B})$ is a non-degenerate $*$-homo\-morphism $f\colon A \to M(B)$ satisfying $\Delta_{B} \circ f = (f\otimes f)\circ \Delta_{A}$.

A $C^{*}$-quantum group is a $C^{*}$-bialgebra that arises from a well-behaved multiplicative unitary as follows \cite{woron:mu-qg-2,soltan:2,woron:mu-qg}.  Suppose that $H$ is a Hilbert space and that $W \in \mathcal{B}(H \otimes H)$ is a multiplicative unitary \cite{baaj:mu} that is manageable or modular \cite{woron:mu-qg,soltan:2}.  Then the spaces
 \begin{align*} 
   A & :=[(\omega\otimes\id_{H})W : \omega\in\mathcal{B}(H)_*] &&\text{and} & \hat {A} &:= [(\id_{H}\otimes\omega)W : \omega\in\mathcal{B}(H)_*]
  \end{align*}
  are separable, nondegenerate $C^{*}$-subalgebras of $\mathcal{B}(H)$, the unitary $W$ is a multiplier of $\hat{A} \otimes A \subseteq \mathcal{B}(H\otimes H)$, and the formulas
  \begin{align} \label{eq:comult} \Delta(a) &= W(a\otimes 1_{H})W^{*}, & \hat \Delta(\hat a) &= \sigma(W^{*}(1_{H} \otimes \hat a)W)
\end{align}
define comultiplications on $A$ and $\hat{A}$, respectively, such that $(A,\Delta)$ and $(\hat{A},\hat{\Delta})$ become $C^{*}$-bialgebras.  
A $C^{*}$-bialgebra $(A,\Delta)$ is a \emph{$C^{*}$-quantum group} if it arises from a modular multiplicative unitary $W$ as above. 

Let $(A,\Delta)$ be a $C^{*}$-quantum group arising from a unitary $W$ as above.  Denote by $\Sigma$ the flip on $H\otimes H$. Then also the \emph{dual} $\widehat{W}:=\Sigma W^{*} \Sigma$ of $W$ is a modular or manageable multiplicative unitary and the associated $C^{*}$-quantum group is $(\hat{A},\hat{\Delta})$. The latter only depends on $(A,\Delta)$ and not on the choice of $W$, and is called the \emph{dual} of $(A,\Delta)$.  The images of $W$ and $\widehat{W}$ in $M(\hat{A} \otimes A)$ or $M(A \otimes \hat{A})$, respectively, do not depend on the choice of $W$ but only on $(A,\Delta)$. We call them the \emph{reduced bicharacters} of $(A,\Delta)$ and $(\hat{A},\hat{\Delta})$ and denote them by $W^{A}$ and $\widehat{W}^{A}$, respectively. We will need an \emph{anti-Heisenberg pair} for $(A,\Delta)$, which consists of non-degenerate, faithful representations $\pi$ of $A$ and $\hat\pi$ of $\hat{A}$ on a Hilbert space $K$ such that the unitary
 \begin{align}
   \label{eq:right-unitary}
   V:=(\id_{A} \otimes \hat{\pi})(\widehat{W}^{A}) \in M(A \otimes \hat{\pi}(\hat{A})),
 \end{align}
regarded as an element of $M(A \otimes \mathcal{K}(K))$, satisfies 
  \begin{align} \label{eq:anti-heisenberg}
    V(1_{A} \otimes \pi(a)) V^{*} &= (\id_{A} \otimes \pi)\Delta(a) \text{ for all } a\in A;
  \end{align}
see \cite[\S 3]{woron:twisted} and \cite[\S 3.1]{timmermann:maximal} .

Every locally compact quantum group or, more precisely, every reduced $C^{*}$-algebraic quantum group in the sense of Kustermans and Vaes \cite{vaes:1}, is a $C^{*}$-quantum group.

\new

We shall use  regularity of $C^{*}$-quantum groups, which was studied for multiplicative unitaries  in \cite{baaj:mu} and   for  reduced $C^{*}$-algebraic quantum groups  in  \cite[\S 5(b)]{baaj:mu-semi-regular}. We  follow the approach of  \cite[Definition 5.37]{roy:thesis} and  call a $C^{*}$-quantum group $(A,\Delta)$ \emph{regular}  if its reduced bicharacter satisfies $[(\hat{A} \otimes 1_{A})W^{A} (1_{\hat{A}}\otimes A)] = \hat{A} \otimes A$ in $M(\hat{A} \otimes A)$.  This is equivalent to the condition $[ (1_{\hat{A}}\otimes A)W^{A}(\hat{A} \otimes 1_{A})] = \hat{A} \otimes A$, see \cite[proof of Corollary 5.39]{roy:thesis}. For the unitary \eqref{eq:right-unitary}, this translates into
\begin{align} \label{eq:weakly-regular}
  [(1_{A} \otimes \hat{\pi}(\hat{A}))V(A \otimes 1_{\hat{\pi}(\hat{A})})] = A \otimes \hat{\pi}(\hat{A}) \quad \text{in } M(A \otimes \hat{\pi}(\hat{A})).
\end{align}
In \cite{roy:thesis}, this condition is referred to as \emph{weak regularity}. However, every reduced $C^{*}$-algebraic quantum $(A,\Delta)$ is regular in the sense above if and only if it is regular in the sense of \cite[\S 5(b)]{baaj:mu-semi-regular}. One implication is contained in \cite[Proposition 3.6]{baaj:mu}, and the other follows easily from \cite[Proposition 5.6]{baaj:mu-semi-regular}.

\old

A \emph{compact $C^{*}$-quantum group} is, by definition, a unital $C^{*}$-bialgebra $\qG=(A,\Delta)$ that satisfies the cancellation conditions, and is indeed a weakly regular $C^{*}$-quantum group \cite{woronowicz}.  Associated to such a compact quantum group is a rigid $C^{*}$-tensor category of unitary finite-dimensional corepresentations \cite{neshveyev:book}. We denote by $\Irr(\qG)$ the equivalence classes of irreducible corepresentations. Their matrix elements span a dense Hopf subalgebra $\mathcal{O}(\qG)$.  The dual $(\hat{A},\hat{\Delta})$ is called a \emph{discrete $C^{*}$-quantum group}, and the underlying $C^{*}$-algebra $\hat{A}$ is a direct sum of matrix algebras, indexed by $\Irr(\qG)$. \new We also denote the underlying $C^{*}$-algebra $\hat{A}$ of $\dqG$ by $C_{0}(\dqG)$. \old

\subsection*{Strict $*$-homomorphisms of $C^{*}$-algebras}
Recall from \cite[\S 5, Corollary 5.7]{lance} that a $*$-homomorphism $\pi\colon B \to M(C)$ is \emph{strict} if it is strictly continuous on the unit ball, and that in that case, it extends to a $*$-homomorphism $M(B) \to M(C)$ that is strictly continuous on the unit ball.  We denote this extension by $\pi$ again. Using this extension, we define the composition of strict $*$-homomorphisms, which evidently is strict again. Hence, $C^{*}$-algebras with strict $*$-homomorphisms form a category.

Recall that a \emph{corner} of a $C^{*}$-algebra $B$ is a $C^{*}$-subalgebra of the form $pBp$ for some projection $p\in M(B)$. 

Strict $*$-homomorphisms are just non-degenerate $*$-homomorphisms in the usual sense from the domain to a corner of the target. Indeed, if $\pi \colon B \to M(C)$ is a strict $*$-homomorphism, then $p:=\pi(1_{B}) \in M(C)$ is a projection, $pCp \subseteq C$ is a corner, and the co-restriction $\pi\colon B \to M(pCp)$ is non-degenerate. Conversely, given a corner $C_{0} \subseteq C$ and a non-degenerate $*$-homomorphism $\pi \colon B \to M(C_{0})$, we get a strict extension $M(B) \to M(C_{0})$, a natural strict map $M(C_{0}) \to M(C)$ \cite[II.7.3.14]{blackadar:opalgs}, and the composition is a strict $*$-homomorphism.

This description of strict $*$-homomorphisms immediately implies that  the minimal tensor product of strict morphisms is a strict morphism again, and that an embedding of $C^{*}$-algebras $B \hookrightarrow C$ is a  strict $*$-homomorphism if and only if  $B$ is a non-degenerate $C^{*}$-subalgebra of a corner of $C$. We shall call such embeddings \emph{strict}.

In the commutative case, partial morphisms correspond to partially defined continuous maps with clopen domain of definition. Indeed, let $X$ and $Y$ be locally compact Hausdorff spaces.  Then every continuous map $F$ from a clopen subset $D \subseteq Y$ to $X$ induces a strict $*$-homomorphism $F^{*} \colon C_{0}(X) \to M(C_{0}(Y))=C_{b}(Y)$ defined by
  \begin{align*}
    (F^{*}(f))(y) = 0 \text{ if } y\not\in D, \quad (F^{*}(f))(y) = f(F(y)) \text{ if  } y \in D.
  \end{align*}
 Conversely, if $\pi \colon C_{0}(X) \to M(C_{0}(Y))$ is a strict $*$-homomorphism, then $\pi(1_{X})$ is the characteristic function of a clopen subset $D \subseteq Y$ and the corestriction $\pi \colon C_{0}(X) \to M(C_{0}(D))$ is the pull-back along a continuous function $F\colon D \to X$.

 \subsection{The slice map property} \label{subsection:slice} In sections \ref{sec:dilation} and \ref{sec:minimal}, we  need the following property.  
A $C^{*}$-algebra $A$ has the \emph{slice map property} if for every $C^{*}$-algebra $B$ and every $C^{*}$-subalgebra $C \subseteq B$, every $x\in B\otimes A$ satisfying $(\id \otimes \omega)(x) \in C$ for all $\omega \in A^{*}$ lies in $C \otimes A$ \cite{wassermann:slice}. This property holds if $A$ is nuclear, or, more generally, if $A$ has the completely bounded approximation property or the strong operator approximation property; see \cite{zacharias:fubini} for a survey. In particular, this condition holds whenever $(A,\Delta)$ is a discrete quantum group, or, more generally, whenever $(A,\Delta)$ is a reduced $C^{*}$-algebraic quantum group whose dual is amenable \cite[Theorem 3.3]{bedos:amenable}.

\section{Partial coactions of $C^{*}$-bialgebras} \label{sec:partial}

The definition of a partial coaction given for Hopf algebras in  \cite{caenepeel:partial} carries over to $C^{*}$-bialgebras as follows.
\begin{definition} \label{definition:partial-coaction}
 A \emph{partial coaction} of a $C^{*}$-bialgebra $(A,\Delta)$ on a $C^{*}$-algebra $C$ is a strict $*$-homomorphism $\delta\colon C \to M(C\otimes A)$ satisfying the following conditions:
  \begin{enumerate}
  \item $\delta(C)(1_{C} \otimes A) \subseteq C\otimes A$;
    \item $\delta$ is partially coassociative in the sense that
    \begin{align} \label{eq:partial-coaction}
      (\delta \otimes \id_{A})\delta(c) = (\delta(1_{C}) \otimes 1_{A})(\id_{C} \otimes \Delta)\delta(c)
    \end{align}
 for all $c \in C$,
 or, equivalently, the following diagram commutes:
    \begin{align} \label{eq:partial-coassociativity}
      \xymatrix@C=80pt{
C\ar[r]^{\delta} \ar[d]_{\delta} & M(C\otimes A) \ar[d]^{\delta \otimes \id} \\ M(C\otimes A) \ar[r]^(0.45){(\delta(1_{C}) \otimes 1_{A}) (\id_{C} \otimes \Delta)\delta } & M(C\otimes A \otimes A)}
    \end{align}
  \end{enumerate}
\end{definition}
Let $\delta$ be a partial coaction of a $C^{*}$-bialgebra $(A,\Delta)$ on a $C^{*}$-algebra $C$.
For every functional $\omega \in A^{*}$ and every multiplier $T \in M(C)$, we  define a multiplier
\begin{align*}
  \omega \triangleright T := (\id_{C} \otimes \omega)\delta(T) \in M(C),
\end{align*}
where we use the fact that we can write $\omega=a \upsilon$ or $\omega= \upsilon' a'$ with $a,a'\in A$ and $\upsilon,\upsilon' \in A^{*}$ by Cohen's factorization theorem. 

Let  $c\in C$ and $\omega\in A^{*}$. Then conditions (1) and (2) in Definition \ref{definition:partial-coaction} imply  $ \omega \triangleright c \in C$ and
\begin{align} \label{eq:delta-act}
  \delta(\omega \triangleright c) = (\id_{C} \otimes \id_{A} \otimes \omega)(\delta\otimes \id_{A})\delta(c) = \delta(1_{C}) (\id_{C} \otimes \id_{A} \otimes \omega)(\id_{C} \otimes \Delta)\delta(c).
\end{align}
In particular, for every character $\chi \in A^{*}$,
\begin{align}
  \label{eq:character-act}
  \chi \triangleright (\omega \triangleright c) = (\chi \triangleright 1_{C})(\id_{C} \otimes (\chi\otimes \omega)\Delta)\delta(c) = (\chi \triangleright 1_{C})(\chi \omega \triangleright c).
\end{align}

The following conditions on a partial coaction are straightforward generalizations of the corresponding conditions on coactions, and will play an equally important role:
\begin{definition}
We say that a partial coaction $\delta$ of a $C^{*}$-bialgebra $(A,\Delta)$ on a $C^{*}$-algebra $C$
\begin{itemize}
\item satisfies the \emph{Podle\'s condition}  if  $[\delta(C)(1_{C} \otimes A)]=[\delta(1_{C})(C\otimes A)]$;
\item  is \emph{weakly continuous} if $[A^{*} \triangleright C] = C$;
 \item is \emph{counital} if $(A,\Delta)$ has a counit $\varepsilon$ and  $(\id_{C} \otimes \varepsilon)\circ \delta = \id$.
\end{itemize}
\end{definition} 
\begin{remark} \label{remark:weak-continuity-hahn-banach}
 If $\delta$ is a partial coaction as above and $X \subseteq A^{*}$ is a subset that separates the points of $A$, then a standard application of the Hahn-Banach theorem shows that $[X \triangleright C]=[A^{*} \triangleright C]$.
\end{remark}
 Every counital partial coaction evidently is  weakly continuous.

\new 
A coaction satisfying the Podle\'s condition is automatically weakly continuous, and is usually called \emph{(strongly) continuous}. For partial coactions,  this implication does no longer hold in general, and so we avoid this terminology. \old
\begin{lemma}
  Let $\delta$ be a partial coaction of a $C^{*}$-bialgebra $(A,\Delta)$ on a $C^{*}$-algebra $C$ that satisfies the Podle\'s condition. Then:
\begin{enumerate}
\item $\delta$ is weakly continuous if and only if  $[(A^{*}\triangleright 1_{C})C]=C$;
\item $\delta$ is counital if and only if  $(A,\Delta)$ has a counit $\varepsilon$ and  $\varepsilon \triangleright 1_{C}=1_{C}$.
\end{enumerate}
  \end{lemma}
\begin{proof}
  (1) By assumption, the closed linear span of all elements of the form $a\omega \triangleright c = (\id_{C} \otimes \omega)(\delta(c)(1_{C} \otimes a))$, where $\omega \in A^{*}$, $a\in A$ and $c\in C$, is equal to the closed linear span of all elements of the form $(\id_{C} \otimes \omega)(\delta(1_{C})(c \otimes a)) = (a\omega \triangleright 1_{C})c$.  Now, use  Cohen's factorization theorem.

  (2) If   $\varepsilon \triangleright 1_{C}=1_{C}$, then  elements of the form $a\varepsilon \triangleright c$, where $a\in A$ and $c\in C$, are linearly dense in $C$, and for every $\omega \in A^{*}$ and $c\in C$, \eqref{eq:character-act} implies
  $\varepsilon \triangleright (\omega \triangleright c) =      1_{C} \cdot (\omega \triangleright c)$.
\end{proof}

For regular reduced $C^{*}$-algebraic quantum groups, weakly continuous coactions
automatically satisfy the Podle\'s condition \cite[Proposition
5.8]{baaj:mu-semi-regular}. More generally, we show the following:
\begin{proposition} \label{proposition:regular-weak-strong} Let $(A,\Delta)$ be a regular $C^{*}$-quantum group. Then every weakly continuous partial coaction of $(A,\Delta)$ satisfies the Podle\'s condition.
\end{proposition}
\begin{proof}
  We proceed similarly as in the proof of \cite[Proposition 5.8]{baaj:mu-semi-regular}, and use an anti-Heisenberg pair $(\pi,\hat\pi)$ for $(A,\Delta)$ on some Hilbert space $K$ and the unitary $V$ in \eqref{eq:right-unitary}. 

Let $\delta$ be a weakly continuous partial coaction of $(A,\Delta)$ on a $C^{*}$-algebra $C$.  
By  \eqref{eq:delta-act} and Remark \ref{remark:weak-continuity-hahn-banach},
\begin{align*}
   [ \delta(C)(1_{C}& \otimes A) ] = [ \delta(\omega \circ \pi \triangleright C)(1_{C} \otimes A) :  \omega \in \mathcal{B}(K)_{*}] \\
&= [  \delta(1_{C})\cdot (\id_{C} \otimes \id_{A} \otimes \omega\circ \pi)((\id_{C} \otimes \Delta)(\delta(C))) \cdot (1_{C} \otimes A) :  \omega \in \mathcal{B}(K)_{*}].
 \end{align*}
To shorten the notation, let $\delta_{\pi}:=(\id_{C} \otimes \pi)\circ \delta$.  We use the relations \eqref{eq:anti-heisenberg}, \eqref{eq:weakly-regular} and $[\hat\pi(\hat{A}) \mathcal{B}(K)_{*}]=\mathcal{B}(K)_{*}$,  and find
\begin{align*}
[ (\id_{C} \otimes  \id_{A} &\otimes \omega \circ \pi)( (\id_{C} \otimes \Delta)(\delta(C))(1_{C} \otimes A \otimes 1_{A})) :  \omega \in \mathcal{B}(K)_{*}] \\
&=   [ (\id_{C} \otimes \id_{A} \otimes \omega)( V_{23}\delta_{\pi}(C)_{13}V^{*}_{23}( A \otimes \hat\pi(\hat{A}))_{23}) : \omega \in \mathcal{B}(K)_{*}]  \\
&=[ (\id_{C} \otimes \id_{A} \otimes \omega)(V_{23} \delta_{\pi}(C)_{13}(A \otimes \hat\pi(\hat{A}))_{23}) :  \omega \in \mathcal{B}(K)_{*}]  \\
&=[ (\id_{C} \otimes \id_{A} \otimes \omega)((1_{A} \otimes \hat\pi(\hat{A}))_{23}V_{23}(A\otimes 1_{K})_{23}\delta_{\pi}(C)_{13} ) :   \omega \in \mathcal{B}(K)_{*}]  \\
&=[ (\id_{C} \otimes \id_{A} \otimes \omega)((A \otimes \hat\pi(\hat{A}))_{23}\delta_{\pi}(C)_{13}) :  \omega \in \mathcal{B}(K)_{*}]  \\
&= [A^{*} \triangleright C] \otimes  A,
\end{align*}
whence $[  \delta(C)(1_{C} \otimes A) ] = [\delta(1_{C})(C \otimes A)]$.
\end{proof}

Partial coactions on $\C$ correspond to certain projections:
\begin{lemma}
  Partial coactions of a $C^{*}$-bialgebra  $(A,\Delta)$ on $\C$ correpond bijectively with  projections $p \in M(A)$ satisfying
  \begin{align} \label{eq:projection-cocycle}
   (p \otimes 1_{A})\Delta(p)=p \otimes p.
  \end{align}
\end{lemma}
\begin{proof}
 Projections $p\in M(A)$ correspond to strict $*$-homomorphisms $\delta\colon \C \to M(\C \otimes A)  \cong  M(A)$ via $p=\delta(1)$, and under this correspondence, $(\delta \otimes \id_{A})\delta(\lambda) = \lambda \otimes p \otimes p$ and
 $(\delta(1) \otimes 1_{A})(\id_{\C} \otimes \Delta)(\delta(\lambda)) = \lambda \otimes (p \otimes 1_{A})\Delta(p)$.
\end{proof}
\new Note that if $(A,\Delta)$ is co-commutative, for example, if $A=C^{*}(G)$ or $A=C^{*}_{r}(G)$ for a locally compact group $G$, then \eqref{eq:projection-cocycle}  just means that $p$ is \emph{group-like} in the sense that $(p \otimes 1_{A})\Delta(p) = p\otimes p = (1_{A}\otimes p)\Delta(p)$. Group-like projections were also studied in connection with idempotent states, see \cite[\S2]{MR2543634}.  Elementary examples related to groups are as follows. \old
  \begin{example} 
Let $G$ be a locally compact group.  
\begin{enumerate}
\item Consider the $C^{*}$-bialgebra $(C_{0}(G),\Delta)$.  A projection $p\in M(C_{0}(G))$ is just the characteristic function of a clopen subset $H \subseteq G$, and satisfies \eqref{eq:projection-cocycle} if and only if $p(g)p(gg')=p(g)p(g')$ for all $g,g' \in G$, that is, if and only if $H \subseteq G$ is a subgroup. Thus, partial coactions of $(C_{0}(G),\Delta)$ on $\C$ correspond to open subgroups of $G$.
  \item Consider the reduced group $C^{*}$-bialgebra $(C^{*}_{r}(G),\Delta)$.  For every finite normal subgroup $N\subseteq G$, the sum $p=\sum_{g\in N} \lambda_{g}$ is a central projection in $M(C^{*}_{r}(G))$ satisfying \eqref{eq:projection-cocycle}, where  $\lambda_{g}$ denotes the left translation by $g\in G$.
\new More information on group-like projections in  $C^{*}_{r}(G)$ and $C^{*}(G)$  can be found in \cite[Proposition 7.6]{MR2437092} and \cite{MR0281843}.    \old
\end{enumerate}
  \end{example}

Every central projection satisfying \eqref{eq:projection-cocycle} gives rise to a 
quotient $C^{*}$-bialgebra $(A_{p},\Delta_{p})$ of $(A,\Delta)$ whose coactions can be regarded as partial coactions of $(A,\Delta)$:
\begin{lemma}
  Suppose that $(A,\Delta)$ is a $C^{*}$-bialgebra with a central projection $p\in M(A)$ satisfying \eqref{eq:projection-cocycle}. Let $A_{p}=pA$ and define $\Delta_{p} \colon A_{p} \to M(A_{p} \otimes A_{p})$ by $a \mapsto (p\otimes p)\Delta(a)$.  Then $(A_{p},\Delta_{p})$ is a $C^{*}$-bialgebra, the map $A \to A_{p}$, $a \mapsto pa$, is a morphism of $C^{*}$-bialgebras, and
 every coaction of
    $(A_{p},\Delta_{p})$ can be regarded as a partial coaction of
    $(A,\Delta)$.
  \end{lemma}
\begin{proof}
 All of these assertions are easily verified, for example, if $\delta$ is a coaction of $(A_{p},\Delta_{p})$ on a $C^{*}$-algebra $C$, then  for all $c\in C$,
\begin{multline*}
(\delta(1_{C}) \otimes 1_{A})  (\id_{C} \otimes \Delta)\delta(c) = (1_{C} \otimes p \otimes   1_{A} )(1_{C} \otimes \Delta)((1_{C} \otimes p)\delta(c)) \\ 
=( 1_{C} \otimes p \otimes p)(\id_{C} \otimes \Delta)\delta(c)  
=( 1_{C} \otimes \Delta_{p})\delta(c) 
=   (\id_{C} \otimes \delta)\delta(c).  \qedhere
\end{multline*}
\end{proof}
\begin{example}
Let $\qG=(A,\Delta)$ be a discrete quantum group, so that $A$ is a $c_{0}$-sum of matrix algebras indexed by $ \Irr(\hat{\qG})$. Consider a central projection $p\in M(A)$ supported on   $\mathcal{J} \subseteq \Irr(\hat{\qG})$.  Then $(p\otimes 1)\Delta(p)=p \otimes p$ if and only if the following condition holds:
  \begin{align} \label{eq:p-discrete}
 \text{If $\alpha \in \mathcal{J}$, $\beta,\gamma \in \Irr(\hat\qG)$ and $\alpha \otimes \beta$ contains $\gamma$, then $\beta \in \mathcal{J}$ if and only if $\gamma \in \mathcal{J}$.}    
  \end{align}
 If   $(A_{p},\Delta_{p})$ is a discrete quantum subgroup of $(A,\Delta)$, then $\mathcal{J}$ is closed under taking duals and summands of tensor products, and then Frobenius duality implies \eqref{eq:p-discrete}. Conversely, suppose that  \eqref{eq:p-discrete} holds.  Taking  $\gamma=\alpha$, we see that $\mathcal{J}$ contains the trivial representation,  and taking this for $\gamma$, we see that $\mathcal{J}$ contains the dual of $\alpha$. Thus, finite sums of representations in $\mathcal{J}$ form a rigid tensor subcategory, and  $(A_{p},\Delta_{p})$ is a discrete quantum subgroup of $(A,\Delta)$. 
\end{example}

\section{The relation to partial actions of groups}
\label{sec:groups}

  We now relate partial actions of a (discrete) group $\Gamma$ to counital partial coactions of the $C^{*}$-bialgebra $C_{0}(\Gamma)$. Recall  that a \emph{partial action} of  $\Gamma$ on a $C^{*}$-algebra $C$ is a family $(D_{g})_{g\in \Gamma}$ of closed ideals of $C$ together with a family $(\theta_{g})_{g\in \Gamma}$ of isomorphisms $\theta_{g} \colon D_{g^{-1}} \to D_{g}$ such that 
  \begin{itemize}
  \item[(G1)] $D_{e} = C$ and $\theta_{e} = \id_{C}$, where $e\in \Gamma$ denotes the unit,
  \item[(G2)] $\theta_{g^{-1}}\theta_{g}\theta_{h} = \theta_{g^{-1}}\theta_{gh}$ and $\theta_{g}\theta_{h}\theta_{h^{-1}} = \theta_{gh}\theta_{h^{-1}}$  for all $g,h \in \Gamma$  as partially defined maps;
  \end{itemize}
see \cite{exel:book,mcclanahan}. We show that partial coactions of $C_{0}(\Gamma)$ correspond to partial actions of $\Gamma$ as above, where each ideal $D_{g}$  is a direct summand, and adopt the following terminology:
\begin{definition} 
  A \emph{disconnected partial action} of $\Gamma$ on a $C^{*}$-algebra $C$ is given by a family $(p_{g})_{g\in \Gamma}$ of central projections  in $M(C)$ and a family $(\theta_{g})_{g\in \Gamma}$ of isomorphisms $\theta_{g} \colon p_{g^{-1}}C\to p_{g}C$ such that $((p_{g}C)_{g\in \Gamma}, (\theta_{g})_{g\in \Gamma},)$ is a partial action.
\end{definition}
\new \begin{remark}
  \begin{enumerate}
  \item Let $X$ be a locally compact Hausdorff space. Then partial actions of $\Gamma$ on $C_{0}(X)$ correspond bijectively to partial actions of $\Gamma$ on $X$ \cite[Corollary 11.6]{exel:book}, and a partial action on $C_{0}(X)$ is disconnected if and only if for every group element $g \in \Gamma$, the domain of definition of its action on $X$ is not only open but also closed. This condition also implies that the partial action on $X$ admits a globalization that is Hausdorff \cite[Proposition 5.7]{exel:book}.
  \item A partial action of $\Gamma$ on an algebra $C$ admits a globalization if and only if for every group element $g\in \Gamma$, its domain of definition is not just a two-sided ideal of $C$ but also unital, that is, a direct summand \cite[Theorem 6.13]{exel:book}.
  \end{enumerate}
\end{remark} \old
We denote  by $C_{b}(\Gamma;C)$ the $C^{*}$-algebra of norm-bounded $C$-valued functions on $\Gamma$, and identify this $C^{*}$-algebra with a subalgebra of $M(C \otimes C_{0}(\Gamma))$ in the canonical way. For each $g\in \Gamma$, we denote by $\ev_{g} \in C_{0}(\Gamma)^{*}$ the evaluation at $g$.
\begin{proposition} \label{proposition:partial-group-action}
  Let $\Gamma$ be a group and  let $C$ be a $C^{*}$-algebra.
\begin{enumerate}
\item Let $\delta$ be a counital partial coaction of $C_{0}(\Gamma)$ on $C$. Then the projections
  \begin{align*}
   p_{g}:=\ev_{g} \triangleright 1_{C}
  \end{align*}
are central and the maps $\theta_{g} \colon p_{g^{-1}}C \to p_{g}C$ given by
\begin{align*}
  \theta_{g}(c):=  \ev_{g} \triangleright c
\end{align*}
form a disconnected partial action of $\Gamma$ on $C$.
\item Let $((p_{g})_{g\in \Gamma},(\theta_{g})_{g\in \Gamma})$ be a disconnected  partial action of $\Gamma$ on  $C$. Then the map
  \begin{align*}
   \delta \colon C \to C_{b}(\Gamma;C) \hookrightarrow M(C \otimes C_{0}(\Gamma)) 
  \end{align*}
 defined by
  \begin{align*}
    (\delta(c))(g) &:= \theta_{g}(p_{g^{-1}}c)  \quad (c\in C, \, g\in \Gamma)
  \end{align*}
is a counital partial coaction of $C_{0}(\Gamma)$ on $C$.
  \end{enumerate}
\end{proposition}
\begin{proof}
(1)    For each $g\in \Gamma$,   the map $\Theta_{g} \colon C \to C$ given by $c\mapsto \ev_{g} \triangleright c$
 is a strict endomorphism.
Since $\delta$ is counital, $\Theta_{e}$ is the identity on $C$.  
 Let $g,h \in \Gamma$. Then by \eqref{eq:character-act},
\begin{align} \label{eq:rho-1}
    \Theta_{g}(\Theta_{h}(c)) &= (\ev_{g} \triangleright 1_{C})(\ev_{g}\ev_{h} \triangleright c) = p_{g}\Theta_{gh}(c),
  \end{align}
in particular, 
\begin{align} \label{eq:rho-2}
\Theta_{g}(p_{h}) &= p_{g}p_{gh}, & \Theta_{g}(\Theta_{g^{-1}}(c)) &= p_{g}c, & \Theta_{g^{-1}}(\Theta_{g}(c)) = p_{g^{-1}}c.
\end{align}
Since $\Theta_{g}\circ \Theta_{g^{-1}}$  is a $*$-homomorphism, the second equation implies $p_{g}c = cp_{g}$ for all $c\in C$, that is, $p_{g}$ is central and $D_{g}:=p_{g}C$ is a direct summand of $C$. The second and third equations imply that $\Theta_{g}$ and $\Theta_{g^{-1}}$ restrict to mutually inverse isomorphisms
\begin{align*}
  D_{g^{-1}} \stackrel{\theta_{g}}{\underset{\theta_{g^{-1}}}{\rightleftarrows}} D_{g}.
\end{align*}
It remains to show that $\theta_{g^{-1}}\theta_{gh}=\theta_{g^{-1}}\theta_{g}\theta_{h}$. But the relations \eqref{eq:rho-1} and \eqref{eq:rho-2} imply that
\begin{align*}
  (\Theta_{g^{-1}}\circ \Theta_{gh})(c) = p_{g^{-1}}\Theta_{gh}(c) = (\Theta_{g^{-1}} \circ \Theta_{g} \circ \Theta_{h})(c) 
\end{align*}
for all $c\in C$,
and that the compositions  $\theta_{g^{-1}}\theta_{gh}$ and $\theta_{g^{-1}}\theta_{g}\theta_{h}$ have the domain
\begin{align*}
  \Theta_{h^{-1}g^{-1}}(p_{g})C = p_{h^{-1}g^{-1}}p_{h^{-1}}C =  \Theta_{h^{-1}}(p_{g^{-1}})C.
\end{align*}

(2) For each $g \in \Gamma$, denote by $\delta_{g} \in C_{0}(\Gamma)$ the characteristic function of $\{g\} \subset \Gamma$.  Then 
\begin{align*}
  \delta(c)(1_{C} \otimes \delta_{g}) = \theta_{g}(p_{g}c) \otimes \delta_{g} \quad (g\in \Gamma, \, c\in C).
\end{align*}
We conclude that  $\delta(C)(1_{C} \otimes C_{0}(\Gamma))$ is contained in $C \otimes C_{0}(\Gamma)$, and that $\delta$ co-restricts to a non-degenerate $*$-homomorphism from $C$ to
 $q(C \otimes C_{0}(\Gamma))$, where $q = \sum_{g \in \Gamma} p_{g} \otimes \delta_{g}$, so that $\delta$ is strict. To verify that $\delta$  is partially coassociative, it suffices to check that for all $g,h\in \Gamma$ and $c\in C$, the element
\begin{align*}
  (\id_{C} \otimes \ev_{g} \otimes \ev_{h})(\delta \otimes \id_{A})\delta(c) =    \theta_{g}(p_{g^{-1}}\theta_{h}(p_{h^{-1}}c))
\end{align*}
is equal to  the element
\begin{align*}
  (\id_{C} \otimes \ev_{g} \otimes \ev_{h})((\delta(1_{C}) \otimes 1_{A})(\id_{C} \otimes \Delta)\delta(c)) &=
                                                      \theta_{g}(p_{g^{-1}})\theta_{gh}(p_{h^{-1}g^{-1}}c),
\end{align*}
and this follows easily from the definition of a partial action.
\end{proof}

The following example shows that the correspondence between partial coactions of $C_{0}(\Gamma)$ and partial actions of $\Gamma$ does not easily extend from groups to inverse semigroups.
\begin{example}
  Denote by $\Gamma$ the inverse semigroup consisting of the $2\times 2$-matrices  
  \begin{align*}
    &0, & v&=
  \begin{pmatrix}
    0 & 0 \\ 1 & 0
  \end{pmatrix}, & v^{*} &=
\begin{pmatrix}
    0 & 1 \\ 0 & 0
  \end{pmatrix}, 
& vv^{*} &=
\begin{pmatrix}
    0 & 0 \\ 0 & 1
  \end{pmatrix}, 
& v^{*}v &=
\begin{pmatrix}
    1 & 0 \\ 0 & 0
  \end{pmatrix}
  \end{align*}
  with matrix multiplication as composition. Then $C(\Gamma)$ is a $C^{*}$-bialgebra with respect to the transpose $\Delta$ of the multiplication. For $x\in \Gamma$, define $\delta_{x} \in C(\Gamma)$ by $y\mapsto \delta_{x,y}$. Then, for example,
  \begin{align*}
\Delta(\delta_{v^{*}v}) &= \delta_{v^{*}}  \otimes  \delta_{v} + \delta_{v^{*}v} \otimes \delta_{v^{*}v}, &  \Delta(\delta_{v}) &=  \delta_{vv^{*}} \otimes \delta_{v} + \delta_{v} \otimes \delta_{v^{*}v}.
\end{align*}
Now, the $*$-homomorphism 
  \begin{align*}
\delta \colon \C^{2} \to \C^{2} \otimes C(\Gamma), \quad (\alpha,\beta) \mapsto (\alpha,0) \otimes \delta_{v^{*}v} + (0,\alpha) \otimes \delta_{v},
  \end{align*}
is a partial coaction. Indeed, for all $\alpha,\beta\in \C$,
\begin{align*}
(\delta \otimes \id_{C(\Gamma)})\delta((\alpha,\beta)) &= (\alpha,0) \otimes \delta_{v^{*}v} \otimes \delta_{v^{*}v} + (0,\alpha) \otimes \delta_{v} \otimes \delta_{v^{*}v}
\end{align*}
is equal to the product of 
\begin{align*}
  \delta((1,0)) \otimes 1_{C(\Gamma)} = (1,0) \otimes \delta_{v^{*}v} \otimes 1_{C(\Gamma)} + (0,1) \otimes \delta_{v} \otimes 1_{C(\Gamma)}
\end{align*}
with
\begin{multline*}
(\id_{\C^{2}} \otimes \Delta)\delta((\alpha,\beta)) 
=  (\alpha,0) \otimes (\delta_{v^{*}v} \otimes \delta_{v^{*}} + \delta_{v^{*}} \otimes \delta_{vv^{*}})  \\ + (0,\alpha) \otimes (\delta_{v} \otimes \delta_{v^{*}v} + \delta_{vv^{*}} \otimes \delta_{v}).
\end{multline*}
But the maps $\Theta_{w} := (\id \otimes \ev_{w})\circ \delta$, where $w\in \Gamma$, are given by
\begin{align*}
  \Theta_{0} &= \Theta_{v^{*}} = \Theta_{vv^{*}} = 0, & \Theta_{v}((\alpha,\beta)) &= (0,\alpha), & \Theta_{v^{*}v}((\alpha,\beta))&= (\alpha,0);
\end{align*}
in particular, $\Theta_{v}\Theta_{v^{*}}\Theta_{v} = 0$ and $\Theta_{v}\Theta_{v^{*}v} = \Theta_{v}$.
\end{example}

\section{Partial coactions of discrete and of compact $C^{*}$-quantum groups} \label{sec:hopf}

Let $\qG=(A,\Delta)$ be a compact $C^{*}$-quantum group and denote by $\mathcal{O}(\qG) \subseteq A$ the dense Hopf subalgebra of matrix elements of finite-dimensional corepresentations. We now relate  partial (co)actions of $\qG$ and of the discrete dual $\hat{\qG}$ to partial coactions and partial actions of the  Hopf algebra $\mathcal{O}(\qG)$, respectively. Note that $(A,\Delta)$ and $(\hat{A},\hat{\Delta})$ are  regular, so that weakly continuous partial coactions automatically satisfy the Podle\'s condition by Proposition \ref{proposition:regular-weak-strong}.

Recall that a \emph{partial action} of a Hopf algebra $H$ on a unital algebra $C$ is a map
\begin{align*}
  H\otimes C \to C, \quad h \otimes c \mapsto  h\triangleright c,
\end{align*}
satisfying the following conditions:
\begin{itemize}
\item[(H1)] $1_{H} \triangleright c = c$ for all $c\in C$;
\item[(H2)] $h \triangleright (cd) = (h_{(1)} \triangleright c) (h_{(2)} \triangleright d)$ for all $h\in H$ and $c,d\in C$;
\item[(H3)] $h \triangleright (k \triangleright c) = (h_{(1)}\triangleright 1_{C})(h_{(2)}k \triangleright c)$ for all $h,k \in H$ and $c\in C$;
\end{itemize}
see \cite{caenepeel:partial}, and that such a partial action is \emph{symmetric} if additionally
\begin{itemize}
\item[(H4)]  $h \triangleright (k \triangleright c) = (h_{(1)}k \triangleright c)(h_{(2)} \triangleright 1_{C})$ for all $h,k \in H$ and $c\in C$;
\end{itemize}
see \cite{alves:representations}.  \new If additionally $h \triangleright 1_{C}=\varepsilon(h)$ for all $h\in H$, we have a genuine action; in that case, (H3) and (H4) reduce to $h \triangleright (k \triangleright c) = hk \triangleright c$. \old

Recall that the $C^{*}$-algebra  $\hat{A}$ of the discrete $C^{*}$-quantum group $\dqG$ is a $c_{0}$-direct sum of matrix algebras $\hat{A}_{\alpha}$ indexed by $\alpha \in \Irr(\qG)$.  The Hopf algebra $\mathcal{O}(\qG)$ can  be identified with the subspace of all functionals $\omega \in \hat{A}^{*}$ that vanish on $\hat{A}_{\alpha}$ for all but finitely many $\alpha \in \Irr(\qG)$, and  then 
\begin{align*}
  \Delta(\omega)(\hat{a} \otimes \hat{b}) = \omega(\hat{a}\hat{b}) \quad \text{and} \quad
  (\upsilon \omega)(\hat{a}) = (\upsilon \otimes \omega)(\hat{a})
\end{align*}
for all $\upsilon,\omega \in \mathcal{O}(\qG)$ and  $\hat{a},\hat{b} \in \hat{A}$.
\begin{theorem} \label{theorem:algebraic-discrete}
Let $\qG=(A,\Delta)$   be a compact quantum group  and let $\delta$  be  a counital partial coaction of the discrete dual $\dqG=(\hat{A},\hat{\Delta})$ on a unital $C^{*}$-algebra $C$. Then the formula
\begin{align*}
  \upsilon \otimes c \mapsto \upsilon \triangleright c =(\id_{C} \otimes \upsilon)(\delta(c)) \quad (\upsilon \in \mathcal{O}(\qG), \, c\in C)
\end{align*}
defines a symmetric partial action of the Hopf algebra $\mathcal{O}(\qG)$ on $C$.
\end{theorem}
\begin{proof}
Condition (H1)   holds because the unit of $\mathcal{O}(\qG)$, regarded as a functional on $\hat{A}$, is the counit. Let $\upsilon,\omega \in \mathcal{O}(\qG)$ and $c,d\in C$. Choose  central projections $p,q \in \hat{A}$ such that $\upsilon(p\hat{a}) = \upsilon(\hat{a})$, $\omega(\hat{a})=\omega(q\hat{a})$ and $\upsilon_{(1)}(\hat{a})\upsilon_{(2)}(\hat{b}) = \upsilon_{(1)}(p\hat{a})\upsilon_{(2)}(p\hat{b})$ for all $\hat{a},\hat{b}\in \hat{A}$. Then
\begin{align*}
  \upsilon \triangleright cd =  (\id_{C} \otimes \upsilon)((1_{C} \otimes p)\delta(c) \delta(d)(1_{C} \otimes p)).
\end{align*}
Since $(1_{C} \otimes p)\delta(c)$ and $\delta(d)(1_{C} \otimes p)$ are contained in the  tensor product of $C$ with the finite-dimensional $C^{*}$-algebra $p\hat{A}+q\hat{A}$,  this expression  is equal to
\begin{align*}
   (\id \otimes \upsilon_{(1)})((1 \otimes p)\delta(c)) \cdot (\id \otimes \upsilon_{(2)}) (\delta(d)(1 \otimes p)) = (\upsilon_{(1)} \triangleright c)(\upsilon_{(2)} \triangleright d).
\end{align*}
Thus, condition (H2) is satisfied. Likewise,
\begin{align*}
  \upsilon \triangleright (\omega \triangleright c) &= (\id_{C} \otimes \upsilon \otimes \omega)((\delta \otimes \id_{\hat{A}})\delta(c)) \\
&= (\id_{C} \otimes \upsilon \otimes \omega)((1_{C} \otimes  p \otimes q)(\delta(1_{C}) \otimes 1_{\hat{A}})(\id_{C} \otimes \hat{\Delta})(\delta(c))),
\end{align*}
and a similar argument as above shows that this expression is equal to
\begin{align*}
  (\id_{C} \otimes \upsilon_{(1)})(\delta(1_{C}))( \id_{C} \otimes (\upsilon_{(2)} \otimes \omega) \circ \hat\Delta)(\delta(c)) = (\upsilon_{(1)} \triangleright 1_{C}) \cdot (\upsilon_{(2)}\omega \triangleright c). 
\end{align*}
Therefore, condition (H3) holds as well, and a similar argument proves (H4).
\end{proof}
Next, we consider partial coactions of the compact $C^{*}$-quantum group $(A,\Delta)$, and relate them to partial coactions of the Hopf algebra $\mathcal{O}(\qG)$. Recall that a \emph{partial coaction} of a Hopf algebra $H$ on a unital algebra $C$ is a homomorphism
\begin{align*}
  \delta \colon C \mapsto C \otimes H
\end{align*}
satisfying the following conditions,
\begin{itemize}
\item[(CH1)] $(\delta \otimes \id_{H})(\delta(c)) = (\delta(1_{C}) \otimes 1_{H}) \cdot (\id_{C} \otimes \Delta_{H})(\delta(c))$   for all $c\in C$, and
\item[(CH2)] $(\id_{C} \otimes \varepsilon_{H})(\delta_{0}(c)) = c$ for all $c\in C$;
\end{itemize}
see \cite{caenepeel}.
\begin{theorem}
  Let  $\delta$ be a partial coaction of a compact $C^{*}$-quantum group $\qG=(A,\Delta)$ on a unital $C^{*}$-algebra $C$. Then the following conditions are equivalent:
  \begin{enumerate}
  \item  $\delta$ is weakly continuous, $\delta(1_{C})$ lies in the algebraic tensor product $C\otimes \mathcal{O}(\qG)$ and $(\id_{C} \otimes   \varepsilon)(\delta(1_{C}))=1_{C}$, where $\varepsilon$ denotes the counit of $\mathcal{O}(\qG)$.
  \item  $\delta$ restricts to a partial coaction of $\mathcal{O}(\qG)$ on a unital dense $*$-subalgebra $C_{0}$ of $C$.
 \end{enumerate}
    \end{theorem}
\begin{proof}
Denote by $\mathcal{O}(\dqG) \subseteq \hat{A}$ the algebraic direct sum of the matrix algebras $\hat{A}_{\alpha}$ associated to all $\alpha \in \Irr(\qG)$, and recall that we can canonically identify $\mathcal{O}(\qG)$ with a subspace of $A^{*}$.

  (1)$\Rightarrow$(2): By Remark \ref{remark:weak-continuity-hahn-banach}, the subspace $C_{0}=\mathcal{O}(\dqG) \triangleright C$ of $C$ is dense.  We show that $C_{0} \subseteq C$ is a  subalgebra. Let $c,d\in C$ and $\upsilon,\omega \in \mathcal{O}(\qG)$. Then
\begin{align*}
  (\upsilon \triangleright c)(\omega \triangleright d) = (\id_{C} \otimes \upsilon \otimes \omega)(\delta(c)_{12}\delta(d)_{13}),
\end{align*}
where we use the leg notation on $\delta(c)$ and $\delta(d)$. Now, we find finitely many $\upsilon'_{i},\omega'_{i} \in \mathcal{O}(\dqG)$ such that
\begin{align*}
  \upsilon(a) \omega(b) = \sum_{i}(\upsilon'_{i} \otimes \omega'_{i})((a \otimes 1_{A})\Delta(b))
\end{align*}
for all $a,b\in A$, and then
\begin{align*}
  (\upsilon \triangleright c)(\omega \triangleright d) &= \sum_{i}(\id_{C} \otimes \upsilon'_{i} \otimes \omega'_{i})((\delta(c) \otimes 1_{A})(\id \otimes \Delta)(\delta(d))) \\
&= \sum_{i}(\id_{C} \otimes \upsilon'_{i} \otimes \omega'_{i}) (\delta \otimes \id_{A})((c \otimes 1_{A})\delta(d)) = \sum_{i} \upsilon'_{i} \triangleright (c(\omega'_{i} \triangleright d)) \in C_{0}.
\end{align*}

Next, we show that $\delta(C_{0})$ is contained in the algebraic tensor product $C \otimes \mathcal{O}(\qG)$.   Let $\omega \in \mathcal{O}(\dqG)$ and $c\in C$. 
Since $\mathcal{O}(\qG)$ has a basis of elements $(u^{\alpha}_{i,j})_{\alpha,i,j}$ satisfying $\Delta(u^{\alpha}_{i,j}) = \sum_{k} u^{\alpha}_{ik}\otimes u^{\alpha}_{kj}$ \cite[Proposition 5.1]{woronowicz}, we can find finitely many $\upsilon_{1},\ldots,\upsilon_{n} \in \mathcal{O}(\dqG)$ and $a_{1},\ldots,a_{n} \in \mathcal{O}(\qG)$ such that
\begin{align*}
  (\id_{A} \otimes \omega)(\Delta(b)) = \sum_{i=1}^{n} \upsilon_{i}(b)a_{i}
\end{align*}
for all $b\in \mathcal{O}(\qG)$, and then
\begin{align*}
  \delta(\omega \triangleright c) &= (\id_{C} \otimes  \id_{A} \otimes \omega)(\delta \otimes \id_{A})\delta(c) \\
  &= \delta(1_{C}) (\id_{C} \otimes (\id_{A} \otimes \omega)\Delta)\delta(c) 
 = \delta(1_{C}) \cdot \sum_{i=1}^{n} (\upsilon_{i} \triangleright c) \otimes a_{i}
\end{align*}
lies in the algebraic tensor product of $C$ with $\mathcal{O}(\qG))$. Using a basis for $\mathcal{O}(\dqG)$ consisting of functionals $(\phi^{\alpha}_{i,j})_{\alpha,i,j}$ such that $\phi^{\alpha}_{i,j}(u^{\beta}_{k,l}) = \delta_{\alpha,\beta} \delta_{i,k} \delta_{j,l}$, see \cite[\S 6]{woronowicz}, we see that $\delta(C_{0})$ is contained in the algebraic tensor product $C_{0} \otimes \mathcal{O}(\qG)$.

To finish the proof, note that with $\omega,c$ as above, \eqref{eq:character-act} implies
\begin{align*}
  \varepsilon \triangleright (\omega \triangleright c) 
= (\id_{C} \otimes \varepsilon)(\delta(1_{C})) \cdot (\omega \triangleright c) = \omega \triangleright c.
\end{align*}

(2)$\Rightarrow$(1): Since $C_{0} \subseteq C$ is dense, the unit of $C_{0}$ has to be $1_{C}$, whence $\delta(1_{C})$ lies in the algebraic tensor product $C \otimes  \mathcal{O}(\qG)$ and $(\id \otimes \varepsilon)\delta(1_{C})=1_{C}$. To prove weak continuity, we  show that for every $c\in C_{0}$, there exists some $\omega \in A^{*}$ such that $\omega \triangleright c = c$. So, take $c \in C_{0}$ and write $\delta(c) = \sum_{i=1}^{n} d_{i} \otimes a_{i}$ with $d_{i} \in C_{0}$ and $a_{i} \in \mathcal{O}(\qG)$.  By Hahn-Banach,  the restriction of $\varepsilon$ to the finite-dimensional subspace of $A$ spanned by $a_{1},\ldots,a_{n}$ extends to a bounded linear functional $\omega \in A^{*}$  that satisfies $\omega \triangleright c = \varepsilon \triangleright c = c$.
\end{proof}

\section{Restriction} \label{sec:restriction-morphisms}

 Like partial actions of groups and partial (co)actions of Hopf algebras, partial coactions of $C^{*}$-bialgebras can be obtained from non-partial ones by restriction. 
\begin{definition}
   Let $\delta_{B}$ be a partial coaction of a $C^{*}$-bialgebra $(A,\Delta)$ on a $C^{*}$-algebra $B$.
We call a  $C^{*}$-subalgebra $C\subseteq B$  \emph{weakly invariant} if
\begin{align*}
  \delta_{B}(C)(C \otimes A) \subseteq C\otimes A,
\end{align*}
and \emph{strongly invariant} if the embedding $C \hookrightarrow B$  strict and
  $\delta_{B}(C) \subseteq M(C\otimes A) \subseteq M(B\otimes A)$.
  \end{definition}
Note here that if the embedding $C\hookrightarrow B$ is strict, then  the embedding $C\otimes A \hookrightarrow B\otimes A$ is strict as well and extends to an embedding  $M(C\otimes A) \hookrightarrow M(B\otimes A)$.
\begin{remark} \label{remark:invariant}
\begin{enumerate}
  \item Every ideal $C \subseteq B$  is weakly invariant, but not necessarily strongly invariant.
\item A corner $C\subseteq B$ is strongly invariant if and only if $1_{C} \in M(C) \subseteq M(B)$  is strongly invariant in the sense that
\begin{align*}
  \delta_{B}(1_{C})=\delta_{B}(1_{C})(1_{C}\otimes 1_{A}),
\end{align*}
 as one can easily check. 
    If one thinks of elements of $M(B)$ and $M(B\otimes A)$ as $2\times 2$-matrices with respect to the Peirce decomposition $B=1_{C}B+(1_{B}-1_{C})B$, then strong invariance of $C$ means that  $\delta_{B}(C)$ is contained in the  upper left corner, while weak invariance of $C$  means that the off-diagonal part of $\delta_{B}(C)$  vanishes.
  \end{enumerate}
\end{remark}
\begin{example} 
  Suppose that $\delta_{B}$ is the partial coaction corresponding to a disconnected partial action $((p_{g})_{g\in \Gamma}, (\theta_{g})_{g\in \Gamma})$ of a discrete group $\Gamma$ on a $C^{*}$-algebra $B$ as in Proposition \ref{proposition:partial-group-action}, and that $C \subseteq B$ is a direct summand. Then $C$ is automatically weakly invariant, but strongly invariant if and only if $\theta_{g}(p_{g^{-1}}C) \subseteq C$ for all $g\in \Gamma$.
\end{example}
Evidently, partial coactions can be restricted to strongly invariant $C^{*}$-subalgebras.
 Restriction to  weakly invariant $C^{*}$-subalgebras is a bit more delicate unless the embedding of the $C^{*}$-subalgebra is strict.
\begin{proposition} \label{proposition:restriction}
   Let $\delta_{B}$ be a partial coaction of a $C^{*}$-bialgebra $(A,\Delta)$ on a $C^{*}$-algebra $B$ and let $C \subseteq B$ be   a weakly invariant $C^{*}$-subalgebra. Then:
   \begin{enumerate}
   \item $\delta_{B}$ restricts to a $*$-homo\-morphism $\delta_{C}\colon C \to M(C \otimes A)$.
   \item If the embedding $C\hookrightarrow B$ is strict, then the composition of $\delta_{C}$ with the embedding of $M(C \otimes A)$ into $M(B\otimes A)$ is strict and
     \begin{align*}
       \delta_{C}(c) = \delta_{B}(c)(1_{C} \otimes 1_{A}) \quad (c\in C).
     \end{align*}
   \item  If $\delta_{C}$ is strict, then it is a partial coaction of $(A,\Delta)$ on $C$.
   \end{enumerate}
\end{proposition}
\begin{proof}
  
  (1) This  follows immediately from the definition.

  (2) Suppose that the  embedding $C\hookrightarrow B$ is strict. Then so is its composition with $\delta_{B}$ and  hence also $\delta_{C}$. To prove the formula given for $\delta_{C}(c)$, choose a bounded approximate unit $(u_{\nu})_{\nu}$ for $C$, and note that $\delta_{C}(c)(u_{\nu} \otimes 1_{A}) = \delta_{B}(c)(u_{\nu} \otimes 1_{A})$ converges strictly to $\delta_{C}(c)$ in $M(C\otimes A)$ and to $\delta_{B}(c)(1_{C} \otimes 1_{A})$ in $M(B \otimes A)$.

(3) Let $(u_{\nu})_{\nu}$ be as above and let $c,c' \in C$. Then by definition of $\delta_{C}$,
\begin{align*}
(c' \otimes 1_{A} \otimes 1_{A}) \cdot& (\delta_{C} \otimes \id_{A})(\delta_{C}(c)(u_{\nu} \otimes 1_{A}))  \\
 &=
(c' \otimes 1_{A} \otimes 1_{A}) \cdot (\delta_{C} \otimes \id_{A})(\delta_{B}(c)(u_{\nu} \otimes 1_{A}))  \\
&= (c' \otimes 1_{A} \otimes 1_{A}) \cdot (\delta_{B} \otimes \id_{A})(\delta_{B}(c)) \cdot (\delta_{C}(u_{\nu}) \otimes 1_{A}) \\
&=  (\id_{C} \otimes \Delta)((c' \otimes 1_{A})\delta_{B}(c)) \cdot  (\delta_{C}(u_{\nu}) \otimes 1_{A}) \\
&= (c' \otimes 1_{A} \otimes 1_{A}) \cdot  (\id_{C} \otimes \Delta)(\delta_{C}(c)) \cdot  (\delta_{C}(u_{\nu}) \otimes 1_{A}).
\end{align*}
Since $c'\in C$ was arbitrary, we can conclude that
\begin{align*}
   (\delta_{C} \otimes \id_{A})(\delta_{C}(c)(u_{\nu} \otimes 1_{A}))   = (\id_{C} \otimes \Delta)(\delta_{C}(c)) \cdot  (\delta_{C}(u_{\nu}) \otimes 1_{A}).
\end{align*}
As $\nu$ tends to infinity, $\delta_{C}(c)(u_{\nu} \otimes 1_{A})$ converges strictly to $\delta_{C}(c)$, and since $\delta_{C}$ and  hence also $\delta_{C} \otimes \id_{A}$ are strict, the left hand side converges to $(\delta_{C} \otimes \id_{A})\delta_{C}(c)$ and the right hand side converges to $(\id_{C} \otimes \Delta)(\delta_{C}(c)) (\delta_{C}(1_{C}) \otimes 1_{A})$.
\end{proof}
\begin{remark}
  \begin{enumerate}
  \item As a corollary, a (partial) coaction on a $C^{*}$-algebra $C$ restricts to a partial coaction on every direct summand of $C$ \new because every direct summand is  weakly invariant by Remark \ref{remark:invariant} (1). \old
  \item The restriction $\delta_{C}$ can be strict without the embedding $C \hookrightarrow B$ being  strict, for example, this is the case if $\delta_{B}$ is the trivial coaction $b \mapsto b \otimes 1_{A}$ and $C \subseteq B$ is a closed ideal that is not a direct summand.
  \end{enumerate}
\end{remark}
\begin{example}
Let $\qG=(A,\Delta)$ be a discrete quantum group, so that $A$ is a $c_{0}$-sum of matrix algebras $A_{\alpha}$ with $\alpha \in \Irr(\hat{\qG})$.  Then
for every subset $\mathcal{J} \subseteq \Irr(\hat{\qG})$, the restriction of  $\Delta$ to the $c_{0}$-sum $A_{J}:=\bigoplus_{\alpha\in J} A_{\alpha}$  yields a partial coaction.  But if $\mathcal{J}$ is non-trivial, then $A_{J}$ is not strongly invariant: if $\alpha \not\in \mathcal{J}$ and $\gamma \in \mathcal{J}$, then    $\alpha \otimes  (\alpha^{\dagger} \otimes \gamma)$, where $\alpha^{\dagger}$ denotes the dual of $\alpha$, contains $\gamma$, and hence $\Delta(A_{\gamma})(A_{\alpha} \otimes 1)\neq 0$.
\end{example}

Closely related to the concept of restriction is the notion of a  morphism of partial coactions.
\begin{definition} \label{definition:morphism}
  Let $\delta_{B}$ and $\delta_{C}$ be partial coactions of a $C^{*}$-bialgebra $(A,\Delta)$ on $C^{*}$-algebras $B$ and $C$, respectively.
 A \emph{strong morphism} from $\delta_{C}$ to $\delta_{B}$ is a strict $*$-homo\-morphism $\pi \colon C\to M(B)$ satisfying 
  \begin{align*}
    (\pi \otimes \id_{A})\delta_{C}(c) = \delta_{B}(\pi(c)) \quad (c\in C).
  \end{align*}
 A \emph{weak morphism} from $\delta_{C}$ to $\delta_{B}$ is a $*$-homomorphism $\pi \colon C \to M(B)$ satisfying
  \begin{align*}
    (\pi \otimes \id_{A})(\delta_{C}(c)(c' \otimes a)) = \delta_{B}(\pi(c))(\pi(c') \otimes a)  \quad (c,c'\in C, \, a\in A).
  \end{align*}
We call such a weak or strong morphism $\pi$  \emph{proper} if $\pi(C) \subseteq B$.
\end{definition} 
Evidently, partial coactions with strong morphisms or with proper weak morphisms as above form categories.
\begin{remark}  
    \begin{enumerate}
  \item Clearly, $\pi$ is a strong or a weak morphism if and only if 
    \begin{align} \label{eq:morphism-act}
      \pi(\omega \triangleright c) = \omega \triangleright \pi(c) \quad \text{or} \quad \pi(\omega \triangleright c)\pi(c') =  (\omega \triangleright \pi(c)) \pi(c')
    \end{align}
respectively, for all $\omega \in A^{*}$  and $c,c'\in C$.
    \item If $\pi$ is a weak or a strong morphism  and proper, then its image  is weakly or strongly invariant, respectively.
  \item  Suppose that $\delta_{B}$ is a partial coaction of $(A,\Delta)$ on a $C^{*}$-algebra  $B$ and that  $C \subseteq B$ is a  $C^{*}$-subalgebra  that is weakly or strongly invariant. If the embedding $C\hookrightarrow B$ is strict, then  this embedding is a weak or a strong morphism with respect to the restriction of $\delta_{B}$ to $C$ defined above.
  \end{enumerate}
\end{remark}
Let us look at the special case of partial coactions associated to disconnected partial group actions.
\begin{proposition} \label{proposition:group-morphism}
Let $B$ and $C$ be two  $C^{*}$-algebras  with  disconnected partial actions $((p_{g})_{g},(\beta_{g})_{g})$ and $((q_{g})_{g},(\gamma_{g})_{g})$, respectively, of a discrete group $\Gamma$. With respect to the associated partial coactions of $C_{0}(\Gamma)$, a strict $*$-homomorphism $\pi \colon B \to M(C)$ is
 a strong morphism if and only if
 \begin{align} \label{eq:group-morphism-strong}
   \pi(p_{g}) = q_{g}\pi(1_{C}) \quad \text{and} \quad \pi \circ \beta_{g} \subseteq \gamma_{g} \circ \pi \quad \text{for all } g\in \Gamma,
 \end{align}
and a weak morphism if and only if
\begin{align}\label{eq:group-morphism-weak}
  \pi(1_{C})\gamma_{g}(q_{g^{-1}}\pi(1_{C})) = \pi(p_{g}) = \gamma_{g}(\pi(p_{g^{-1}})) \ \text{and} \ \pi \circ \beta_{g} \subseteq \gamma_{g} \circ \pi \ \text{for all } g\in \Gamma.
\end{align}
  \end{proposition}
 \begin{proof}
Denote the partial coactions by $\delta_{B}$ and $\delta_{C}$.

(1)    Suppose that $\pi$ is a strong morphism. Then the  definition of $\delta_{B}$ and $\delta_{C}$ implies
\begin{align} \label{eq:group-strong}
  (\pi \circ \beta_{g})(p_{g^{-1}}b) = (\pi \otimes \ev_{g}) \delta_{B}(b) =  (\id_{C} \otimes \ev_{g}) \delta_{C}(\pi(b))  =  \gamma_{g}(q_{g^{-1}} \pi(b))
\end{align}
for all $g \in \Gamma$ and $b\in B$. Taking $b=1_{C}$  or  $b=p_{g^{-1}}$, we conclude  that
\begin{align*}
\gamma_{g}(q_{g^{-1}}\pi(p_{g^{-1}}))   = \pi(p_{g}) = \gamma_{g}(q_{g^{-1}}\pi(1_{C})),
\end{align*}
 in particular, $\pi(p_{g})q_{g}=\pi(p_{g})$. We use this  relation  on the left hand side above, apply $\gamma_{g^{-1}}$, and get  $\pi(p_{g})=q_{g}\pi(1_{C})$. Moreover,  $\pi(p_{g}B) \subseteq q_{g}C$, and \eqref{eq:group-strong} implies    $\pi \circ \beta_{g} \subseteq \gamma_{g} \circ \pi$.

Conversely, the first relation in \eqref{eq:group-morphism-strong} implies $q_{g^{-1}}\pi(1_{C}-p_{g^{-1}}) = 0$, whence both sides in \eqref{eq:group-strong} are zero for all $b \in (1-p_{g^{-1}})B$, and the second relation in  \eqref{eq:group-morphism-strong} implies that \eqref{eq:group-strong} holds for all $b \in p_{g^{-1}}B$. Combined,  \eqref{eq:group-morphism-strong} implies $(\pi \otimes \id)\delta_{B}=\delta_{C} \circ \pi$.

(2) Suppose that $\pi$ is a strict weak morphism.  As in (1), we find that 
\begin{align} \label{eq:group-weak}
  (\pi \circ \beta_{g})(p_{g^{-1}}b) = \pi(1_{C})  \gamma_{g}(q_{g^{-1}} \pi(b))
\end{align} 
 for all  $g\in \Gamma$ and $b\in B$, and similar arguments as in (1)  yield the first equation in \eqref{eq:group-morphism-weak}.  Now, we apply $\gamma_{g^{-1}}$ to this relation and find that
 \begin{align*}
   \gamma_{g^{-1}}(\pi(p_{g})) = \gamma_{g^{-1}}(q_{g}\pi(1_{c})) \pi(1_{C}) = \pi(p_{g^{-1}}).
 \end{align*}
In particular, this relation and \eqref{eq:group-weak} imply the second relation in  \eqref{eq:group-morphism-weak}. 

 Conversely,  \eqref{eq:group-morphism-weak} implies that both sides of \eqref{eq:group-weak} coincide for all $b\in p_{g^{-1}}B$, and that for all $b\in (1_{C}-p_{g^{-1}})B$, 
 \begin{align*}
  \pi(1_{C})\gamma_{g}(q_{g^{-1}}\pi(1_{C}-p_{g^{-1}})) = \pi(p_{g}) -\pi(p_{g})=0,
 \end{align*}
whence both sides of \eqref{eq:group-weak} are zero for all $b\in (1_{C}-p_{g^{-1}})B$. But this implies that $(\pi \otimes \id)\circ \delta_{B} = (\pi(1_{C}) \otimes 1_{A})(\delta_{C} \circ \pi)$. 
 \end{proof}

\new
\section{The Bernoulli shift of a discrete quantum group}
\label{section:bernoulli}

The Bernoulli shift of a discrete group $\Gamma$ is its action on the power set $\mathcal{P}(\Gamma)$, which we identify  with the infinite product $\{0,1\}^{\Gamma}$, by left translation. Restriction to the subsets containing the unit $e_{\Gamma}$ yields an important example of a  partial action. To a discrete quantum group, we now associate a quantum Bernoulli shift and obtain, by restriction, a partial coaction that is initial in a natural sense.

The space $\{0,1\}^{\Gamma}\cong \mathcal{P}(\Gamma)$ parametrizes all maps from $\Gamma$ to $\{0,1\}$ or, equivalently, all subsets of $\Gamma$, which correspond to projections in $M(C_{0}(\Gamma))$.    Given a  discrete quantum group $\qG$, it is natural to define its quantum power set as a universal quantum family of maps from $\qG$ to $\{0,1\}$   in the sense of \cite{soltan:quantum-families} or, equivalently, as the unital $C^{*}$-algebra $C$ that comes with a universal projection in $M(C \otimes C_{0}(\qG))$. However, we need an additional commutativity  assumption.

Let $\qG=(C_{0}(\qG),\Delta)$ be a discrete $C^{*}$-quantum group with counit $\varepsilon$ and compact dual $\dqG$.
\begin{definition}
  Let $C$ be a $C^{*}$-algebra. We call a projection $p \in M(C \otimes C_{0}(\qG))$ \emph{admissible} if in $M(C \otimes C_{0}(\qG) \otimes C_{0}(\qG))$,
  \begin{align} \label{eq:admissible}
  (p \otimes 1) \cdot (\id \otimes \Delta)(p) = (\id \otimes \Delta)(p) \cdot (p\otimes 1).
  \end{align}
   \end{definition}
  \begin{remark}
 For every partial coaction $\delta$ of $C_{0}(\qG)$ on a $C^{*}$-algebra $C$, the projection $\delta(1_{C}) \in M(C \otimes C_{0}(\qG))$ is admissible.
 \end{remark}
 \begin{proposition} \label{proposition:admissible-universal}
  Let $\qG$ be a discrete $C^{*}$-quantum group. Then there exists a unital $C^{*}$-algebra $C(\qB_{\qG})$ with an admissible projection $p \in M(C(\qB_{\qG}) \otimes C_{0}(\qG))$ that is universal in the following sense: for every  $C^{*}$-algebra $C$ with an admissible projection $q \in M(C(\qB_{\qG}) \otimes C_{0}(\qG))$, there exists a unique unital $*$-homomorphism $\pi \colon C(\qB_{\qG}) \to M(C)$ such that $q=(\pi \otimes \id)(p)$.
\end{proposition}
\begin{proof}
    Write $C_{0}(\qG) \cong \bigoplus_{\alpha} I_{\alpha}$, where $\alpha$ varies in $\Irr(\dqG)$ and each $I_{\alpha}$ is a matrix algebra. Choose matrix units $(e^{\alpha}_{ij})_{i,j}$ for each $I_{\alpha}$. Denote by $C(\qB_{\qG})$ the universal unital $C^{*}$-algebra with generators $1$ and  $(p^{\alpha}_{ij})_{\alpha,i,j}$ satisfying the following relations:
    \begin{enumerate}
    \item  the finite sum $p^{\alpha} := \sum_{i,j} p^{\alpha}_{ij} \otimes e^{\alpha}_{ij} $ is a projection for every $\alpha \in \Irr(\dqG)$;
    \item $(p^{\alpha} \otimes 1)(\id \otimes \Delta)(p^{\beta}) = (\id \otimes \Delta)(p^{\beta})(p^{\alpha} \otimes 1)$ for all $\alpha,\beta \in \Irr(\dqG)$.
     \end{enumerate}
     Then the sum $p = \sum_{\alpha} p^{\alpha} \in M(C(\qB_{\qG}) \otimes C_{0}(\qG))$ converges strictly because each summand $p^{\alpha}$ lies in a different summand of $C(\qB_{\qG}) \otimes C_{0}(\qG) \cong \bigoplus_{\alpha} (C(\qB_{\qG}) \otimes I_{\alpha})$ and has norm at most 1. By (1) and (2), this $p$ is an admissible projection, and by construction, $C(\qB_{\qG})$ has the desired universal property by construction.
\end{proof}
We denote by $C_{0}(\qB_{\qG}^{\times}) \subset C(\qB_{\qG})$ the non-unital $C^{*}$-subalgebra generated by all $p^{\alpha}_{i,j}$. 
\begin{example} \label{example:bernoulli-classical}
If  $\qG$ is a classical discrete group $\Gamma$, we can identify $C(\qB_{\Gamma})$ with $C(\{0,1\}^{\Gamma})$. Indeed, in that case, $\Irr(\hat{{\Gamma}})$ can be identified with $\Gamma$ so that $C(\qB_{\Gamma})$ is  generated by $1$ and a family of projections $p^{\gamma}$, where $\gamma \in \Gamma$. Denote by $\delta_{\gamma} \in C_{0}(\Gamma)$  the Dirac delta function at $\gamma\in \Gamma$. Then  $p=\sum_{\gamma} p^{\gamma} \otimes \delta_{\gamma}$ and
the admissibility condition takes the form
\begin{align*}
 \left[\sum_{\gamma}p^{\gamma} \otimes \delta_{\gamma} \otimes 1, \sum_{\gamma,\gamma'} p^{\gamma\gamma'} \otimes \delta_{\gamma} \otimes \delta_{\gamma'}\right] = 0 
\end{align*}
 or, equivalently, $[p^{\gamma},p^{\gamma''}]=0$ for all $\gamma,\gamma'' \in \Gamma$. Thus, $C(\qB_{\Gamma})$ is commutative. Therefore, the map that sends $p^{\gamma}$ to the projection of $\{0,1\}^{\Gamma}$ onto the $\gamma$th component induces an isomorphism $C(\qB_{\Gamma}) \cong C(\{0,1\}^{\Gamma})$. Under this isomorphism, the $C^{*}$-subalgebra $C_{0}(\qB_{\qG}^{\times})$ corresponds to $C_{0}(\mathcal{P}(\Gamma) \setminus \{\emptyset\})$.
\end{example}
 
The quantum space  $\qB_{\qG}$ comes with a natural action of $\qG$:
\begin{proposition} \label{proposition:bernoulli-coaction}
  There exists a unique coaction $\delta$ of $C_{0}(\qG)$ on $C(\qB_{\qG})$ such that
  \begin{align} \label{eq:bernoulli-coaction}
    (\delta\otimes \id)(p) = (\id \otimes \Delta)(p).
  \end{align}
   This coaction is counital and restricts to a coaction on $C_{0}(\qB_{\qG}^{\times})$. 
 \end{proposition}
\begin{proof}
  The projection $q:= (\id \otimes \Delta)(p) \in M((C(\qB_{\qG}) \otimes C_{0}(\qG)) \otimes C_{0}(\qG))$ is admissible because
  \begin{align*}
    (\id \otimes \id \otimes \Delta)(q) = (\id \otimes \Delta^{(2)})(p) =  (\id \otimes \Delta \otimes \id)(\id \otimes \Delta)(p)
  \end{align*}
commutes with $q\otimes 1 = (\id \otimes \Delta \otimes \id)(p \otimes 1)$. The universal property of $p$ yields a unital $*$-homomorphism $\delta \colon C(\qB_{\qG}) \to M(C(\qB_{\qG}) \otimes C_{0}(\qG))$ such that $(\delta \otimes \id)(p)=q=(\id \otimes\Delta)(p)$. We have  $(\delta \otimes \id)\delta = (\id \otimes \Delta)\delta$ because by definition of $\delta$,
\begin{align*}
  ((\delta \otimes \id)\delta \otimes \id)(p) &=  (\delta \otimes \id \otimes \id)(\id \otimes \Delta)(p)   \\
&= (\id \otimes \id \otimes \Delta)(\delta \otimes \id)(p)  \\
&= (\id \otimes \Delta \otimes \id)(\id \otimes  \Delta)(p) = ((\id \otimes \Delta)\delta \otimes \id)(p).  
\end{align*}

Next,  $((\id \otimes \varepsilon) \delta \otimes \id)(p) = (\id \otimes (\varepsilon \otimes \id)\Delta)(p) = p $ and hence $(\id \otimes \varepsilon)\delta = \id$.

Finally, \eqref{eq:bernoulli-coaction} implies that $\delta(p^{\alpha}_{ij})(1 \otimes C_{0}(\qG))  \subseteq  C_{0}(\qB^{\times}_{\qG}) \otimes C_{0}(\qG)$.
\end{proof}
We shall  restrict the coaction $\delta$ to the direct summand of $C(\qB_{\qG})$ that is given by the following projection:
\begin{lemma} \label{lemma:bernoulli-counit}
  The projection $p_{\varepsilon}:=(\id \otimes \varepsilon)(p) \in C(\qB_{\qG})$ is central and $\delta(p_{\varepsilon}) = p$.
\end{lemma}
\begin{proof}
  We apply $\id \otimes \varepsilon \otimes \id$ to \eqref{eq:admissible} and obtain $(p_{\varepsilon} \otimes 1)p=p(p_{\varepsilon} \otimes 1)$. Thus, $p_{\varepsilon}$ commutes with $(\id \otimes \omega)(p) \in C(\qB_{\qG})$ for every $\omega \in C_{0}(\qG)^{*}$ and hence with $C(\qB_{\qG})$.  Moreover, 
  \begin{align*}
    \delta(p_{\varepsilon}) = (\delta \otimes \varepsilon)(p) &= (\id \otimes \id \otimes \varepsilon)(\delta\otimes \id)(p) = (\id \otimes \id \otimes \varepsilon)(\id \otimes \Delta)(p) =p. \qedhere
  \end{align*}
\end{proof}
\begin{example}
  If $\qG$ is a classical group $\Gamma$ (see Example \ref{example:bernoulli-classical}), $\varepsilon$ is evaluation at the unit $e_{\Gamma}$ and $p_{\varepsilon} =p^{e_{\Gamma}}$. Therefore,   restriction of the coaction $\delta$ above to the direct summand $p_{\varepsilon}C(\qB_{\qG})$ of $C(\qB_{\qG})$ corresponds to
restriction of the Bernoulli shift on $\mathcal{P}(\Gamma)$ to the subsets containing the unit $e_{\Gamma}$.  
\end{example}
We can now define the quantum analogue of the partial Bernoulli shift:
\begin{definition}
Let $\qG$ be a discrete $C^{*}$-quantum group and  write $C(\qB_{\qG}^{\varepsilon})$  for the direct summand   $p_{\varepsilon}C(\qB_{\qG})$ of $C(\qB_{\qG})$. Then
  the \emph{partial Bernoulli action} of $\qG$ is the partial coaction $\delta^{\varepsilon}$ of $C_{0}(\qG)$ on $C(\qB_{\qG}^{\varepsilon})$ obtained as the restriction of the coaction $\delta$ as in Proposition \ref{proposition:restriction}, that is,
\begin{align*}
  \delta^{\varepsilon}(b) =  \delta(b)(p_{\varepsilon} \otimes 1) \quad \text{for all  } b \in C(\qB_{\qG}^{\varepsilon}).
\end{align*}
\end{definition}
Proposition \ref{proposition:admissible-universal} immediately implies:
\begin{corollary} \label{corollary:admissible-universal}
    Let $C$ be a $C^{*}$-algebra and let $q \in M(C \otimes C_{0}(\qG))$ be  an admissible projection such that $(\id \otimes \varepsilon)(q)=1_{C} \in M(C)$. Then there exists a unique unital $*$-homomorphism $\pi \colon C(\qB_{\qG}^{\varepsilon}) \to M(C)$ such that $q=(\pi \otimes \id)(p)$.
\end{corollary}
The partial Bernoulli action is initial in the following sense:
\begin{proposition} \label{proposition:bernoulli-initial}
  Let $\delta_{C}$ be a counital partial coaction of $C_{0}(\qG)$ on a $C^{*}$-algebra $C$. Then there exists a unique unital $*$-homomorphism $\pi \colon C(\qB_{\qG}^{\varepsilon}) \to M(C)$ such that
  \begin{align}     \label{eq:bernoulli-initial}
  (\pi \otimes \id)(p(p_{\epsilon} \otimes 1)) = \delta_{C}(1_{C}),
\end{align}
 and this $\pi$ is a strong morphism of partial coactions, that is, $(\pi \otimes \id)\circ \delta^{\varepsilon} = \delta_{C}\circ \pi$.
\end{proposition}
\begin{proof}
The projection $\delta_{C}(1_{C}) \in M(C\otimes C_{0}(\qG))$ is admissible and $\delta_{C}$ is counital. Hence, Corollary \ref{corollary:admissible-universal} yields a unique unital  $*$-homomorphism $\pi \colon C(\qB_{\qG}^{\varepsilon}) \to M(C)$ such that $(\pi \otimes \id)(p) = \delta_{C}(1_{C})$. 
 We show that $(\pi \otimes \id)\circ \delta^{\varepsilon} = \delta_{C} \circ \pi$. First,  \eqref{eq:bernoulli-coaction} and Lemma  \ref{lemma:bernoulli-counit} imply
 \begin{align*}
   (\delta^{\varepsilon} \otimes \id)(p(p_{\varepsilon} \otimes 1)) &= (\delta \otimes \id)(p(p_{\epsilon} \otimes 1)) \cdot (p_{\varepsilon} \otimes 1)  \\
&=(\id \otimes \Delta)(p) \cdot (p \otimes 1) \cdot (p_{\varepsilon} \otimes 1 \otimes 1).
 \end{align*}
We apply $\pi \otimes \id \otimes \id$, use \eqref{eq:bernoulli-initial}, and find that
  \begin{align*}
    ((\pi  \otimes \id)\delta^{\varepsilon} \otimes \id)(p(p_{\varepsilon} \otimes 1)) &= 
(\pi \otimes \Delta)(p(p_{\varepsilon} \otimes 1)) \cdot (\pi \otimes \id)(p(p_{\varepsilon} \otimes 1)) \\
&= (\id \otimes \Delta)(\delta_{C}(1_{C})) \cdot (\delta_{C}(1_{C}) \otimes 1) \\
&= (\delta_{C} \otimes \id)\delta_{C}(1_{C}) \\
&= (\delta_{C} \circ \pi  \otimes \id)(p(p_{\varepsilon} \otimes 1)).
  \end{align*}
But this relation implies $(\pi \otimes \id)\circ \delta^{\varepsilon} = \delta_{C} \circ \pi$.
\end{proof}
This partial Bernoulli shift will be studied further in a forthcoming article. In particular, the partial coaction $\delta^{\varepsilon}$ should give rise to a partial crossed product that can be regarded as a quantum counterpart to the partial group algebra of a discrete group, see \cite[\S10]{exel:book}, and as a $C^{*}$-algebraic counterpart to the Hopf algebroid $H_{\text{par}}$ associated to a Hopf algebra $H$  in \cite{alves:representations}.
\old

\section{Dilations}
\label{sec:dilation}

Let $(A,\Delta)$ be a $C^{*}$-bialgebra.  Given a partial coaction of $(A,\Delta)$, a natural and important question is whether it can be identified with as the restriction of a coaction to a weakly invariant $C^{*}$-subalgebra as in Proposition \ref{proposition:restriction}. 
\begin{definition} \label{definition:dilation}
  Let $\delta_{C}$ be a partial coaction of $(A,\Delta)$ on a $C^{*}$-algebra $C$.  A \emph{dilation} of $\delta_{C}$ consists of a $C^{*}$-algebra $B$, a coaction $\delta_{B}$ of $(A,\Delta)$ on $B$, and an embedding $\iota \colon C \hookrightarrow B$ that is a weak morphism from $\delta_{C}$ to $\delta_{B}$, that is, satisfies
  \begin{align*}
   \delta_{B}(\iota(c))(\iota(c') \otimes a) = (\iota \otimes \id_{A})(\delta_{C}(c)(c' \otimes a)) \quad (c,c'\in C,\, a\in A).
  \end{align*}
\end{definition}
 \begin{example}[Disconnected partial actions of groups]
Let $C$ be a $C^{*}$-algebra with a disconnected partial action  $((p_{g})_{g},(\theta_{g})_{g})$  of a discrete group $\Gamma$, and consider the associated partial coaction $\delta_{C}$ of $C_{0}(\Gamma)$ as in Proposition \ref{proposition:partial-group-action}. 

 A dilation of $\delta_{C}$ is given by a $C^{*}$-algebra $B$ with a coaction of $C_{0}(\Gamma)$, that is, by an action $(\alpha_{g})_{ g \in \Gamma }$ of $\Gamma$ on $B$, and an embedding $C \hookrightarrow B$ that is a weak morphism. Suppose that this embedding is strict. By Proposition \ref{proposition:group-morphism}, it is a weak morphism if and only if
\begin{align*}
  p_{g}=1_{C}\alpha_{g}(1_{C}) \quad \text{and} \quad \theta_{g}=\alpha_{g}\big|_{p_{g}C} \quad (g\in \Gamma).
\end{align*}
In particular,  $1_{C}$ commutes with $\alpha_{g}(1_{C})$ for each $g\in \Gamma$. 
We claim that our partial action coincides with the set-theoretic restriction $((D_{g})_{g},(\alpha_{g}|_{D_{g}})_{g})$ of $\alpha$ to $C$, where $D_{g}=\alpha_{g}(C) \cap C$  for each $g\in \Gamma$. Indeed, for every element $c \in D_{g}$ with $0 \leq c\leq 1_{C}$, we have $c \leq 1_{C}$ and $\alpha_{g}^{-1}(c) \leq 1_{C}$, whence $c \leq \alpha_{g}(1_{C})1_{C} =p_{g}$ and $c \in p_{g}C$. On the other hand, if  $c\in p_{g}C$, then $\alpha_{g^{-1}}(c) = \theta_{g^{-1}}(c) \in C$ and hence $c \in \alpha_{g}(C) \cap C = D_{g}$.

Conversely, suppose that $\alpha$ is an action of $\Gamma$ on a $C^{*}$-algebra $B$ that contains $C$ and that $\alpha$ is a dilation in the usual sense, so that $C \subseteq B$ is an ideal, $p_{g}C = \alpha_{g}(C) \cap C$ and $\theta_{g} = \alpha_{g}|_{p_{g}C}$ for each $g\in \Gamma$. If the embedding $C \subseteq B$ is strict, then $C$ is a direct summand, that is, $C=1_{C}B$, and then $\alpha_{g}(C)\cap C = \alpha_{g}(1_{C})1_{C}$ for each $g\in \Gamma$, so that the coaction $\delta_{B}$ corresponding to $\alpha$ is a dilation of $\delta_{C}$.
 \end{example}
The main question is, of course, which partial coactions do have a dilation. We start with a necessary condition.
\begin{definition}
  We call a partial coaction $\delta_{C}$  of $(A,\Delta)$ on a $C^{*}$-algebra $C$ \emph{regular} if
  \begin{align} \label{eq:regularity}
  (\id_{C} \otimes \Delta)(\delta_{C}(C))  \cdot (1_{C} \otimes 1_{A} \otimes A) \subseteq M(C\otimes A) \otimes A.  
  \end{align}
 \end{definition}
 \begin{example}
   \begin{enumerate}
   \item Every coaction is easily seen to be regular.
   \item The question of regularity arises only if $C$ is non-unital, because every partial coaction on a unital $C^{*}$-algebra is regular.
   \item If $A$ is a direct sum of matrix algebras, for example,  if $(A,\Delta)$ is a discrete quantum group, then every partial coaction of $(A,\Delta)$ is regular.
   \end{enumerate}
 \end{example}
   Regularity is necessary for the existence of a dilation with a strict embedding:
\begin{lemma}
  If  a partial coaction has a dilation $(B,\delta_{B},\iota)$, where $\iota$ is strict, then the partial coaction is regular.
\end{lemma}
\begin{proof}
Suppose that $\delta_{C}$ is  a partial coaction of $(A,\Delta)$ on a $C^{*}$-algebra $C$ with a dilation $(B,\delta_{B},\iota)$.  It suffices to show that the product
   \begin{align*}
     (\iota \otimes \id_{A} \otimes \id_{A})((\id_{C} \otimes \Delta)\delta_{C}(C)) \cdot (1_{B}\otimes 1_{A} \otimes A)
\end{align*}
lies in $M(B\otimes A) \otimes A$. Since $\iota$ is a weak morphism, this product is equal to
\begin{align*}
                                                        (\id_{C} \otimes  \Delta)(\delta_{B}(\iota(C))) \cdot (\iota(1_{C}) \otimes 1_{A} \otimes A),
   \end{align*}
which by coassociativity of $\delta_{B}$ can be rewritten as
\begin{align*}
 (\delta_{B} \otimes \id_{A})(\delta_{B}(\iota(C))(1_{B} \otimes A))  \cdot (\iota(1_{C}) \otimes 1_{A} \otimes 1_{A}),
\end{align*}
and this product  lies in $M(B\otimes A) \otimes A$ because $\delta_{B}(\iota(C))(1_{B}\otimes A) \subseteq B\otimes A$.  
\end{proof}
If $(A,\Delta)$ is a regular $C^{*}$-quantum group, for example, a compact one, and if $\delta_{C}$ is weakly continuous, then regularity of $\delta_{C}$ can be tested on the unit:
\begin{lemma}
  Let $(A,\Delta)$ be a regular $C^{*}$-quantum group and let $\delta_{C}$ be a weakly continuous partial coaction  of $(A,\Delta)$ on a $C^{*}$-algebra $C$ such that
  \begin{align*}
    (\id_{C} \otimes \Delta)(\delta_{C}(1_{C})) \cdot (1_{C} \otimes 1_{A} \otimes A) \subseteq M(C \otimes A) \otimes A.
  \end{align*}
Then $\delta_{C}$ is regular.
  \end{lemma}
\begin{proof}
We use the same notation and a similar argument as in  the proof of Proposition \ref{proposition:regular-weak-strong}.  By \eqref{eq:delta-act},
  \begin{align*}
    (\id_{C} \otimes \Delta)(\delta_{C}(\omega \triangleright c)) =(\id_{C} \otimes \Delta)(\delta_{C}(1_{C}))  \cdot (\id_{C} \otimes \id_{A} \otimes \id_{A} \otimes \omega )  (\id_{C} \otimes \Delta^{(2)})\delta_{C}(c))
  \end{align*}
  for all $\omega \in A^{*}$ and $c\in C$, where $\Delta^{(2)}=(\Delta\otimes \id_{A})\Delta = (\id_{A} \otimes \Delta)\Delta$. Since $\delta_{C}$ is weakly continuous, we can conclude that $[(\id_{C} \otimes \Delta)\delta_{C}(C) \cdot (1_{C} \otimes 1_{A} \otimes A)]$ is equal to the product of $(\id_{C} \otimes \Delta)(\delta_{C}(1_{C}))$ with
  \begin{align*}
     [( \id_{C} \otimes \id_{A} \otimes \id_{A} \otimes \omega)((\id_{C} \otimes \Delta^{(2)})(\delta_{C}(C)) (1_{C}\otimes 1_{A} \otimes A \otimes 1_{A})) : \omega \in A^{*}].
  \end{align*}
Similarly as in the proof of Proposition \ref{proposition:regular-weak-strong}, we  rewrite this  space in the form
  \begin{align*}
 [( \id_{C} \otimes& \id_{A} \otimes \id_{A} \otimes \omega)(V_{34}(\id_{C} \otimes (\id_{A} \otimes \pi) \Delta)(\delta_{C}(C))_{124} V_{34}^{*}(A \otimes \hat\pi(\hat{A}))_{34}) : \omega \in \mathcal{B}(K)_{*}] \\
&=   [( \id_{C} \otimes \id_{A} \otimes \id_{A} \otimes \omega)((A \otimes \hat\pi(\hat{A}))_{34}(\id_{C} \otimes (\id_{A} \otimes \pi) \Delta)(\delta_{C}(C))_{124}) : \omega \in \mathcal{B}(K)_{*}] \\
&=   [( (\id_{C} \otimes \id_{A} \otimes \omega \circ \pi)(\id_{C} \otimes \Delta)\delta_{C}(C)) \otimes A : \omega \in \mathcal{B}(K)_{*} ] \\
&\subseteq M(C \otimes A) \otimes A.
  \end{align*}
Summarising, we find that
\begin{align*}
  (\id_{C} \otimes \Delta)\delta_{C}(C) \cdot (1_{C} \otimes 1_{A} \otimes A) \subseteq (\id_{C} \otimes \Delta)\delta_{C}(1_{C}) \cdot (M(C\otimes A) \otimes A).
\end{align*}
By assumption on $\delta_{C}(1_{C})$, the right hand side lies in $M(C \otimes A) \otimes A$.
\end{proof}

For partial actions of a group $G$ on a  space $X$, a canonical dilation can be constructed as a certain quotient of the product $X\times G$; see \cite{abadie:takai} or \cite[Theorem 3.5, Proposition 5.5]{exel:book}. We now give a dual construction. Although this one will be improved upon  in the next section, we decided to include it for instructive purpose, see also  Example \ref{example:globalization-group}.

From now on, we almost always assume  the $C^{*}$-algebra underlying our $C^{*}$-bialgebra to have the slice map property, which holds, for example, if it is nuclear; see \ref{subsection:slice}. 
\begin{proposition} \label{proposition:dilation-canonical}
Let $\delta_{C}$ be an injective, regular partial coaction of a $C^{*}$-bialgebra $(A,\Delta)$ on a $C^{*}$-algebra $C$ and suppose that $A$ has the slice map property.
   Denote by $C \boxtimes A \subseteq M(C\otimes A)$ the subset of all $x$ satisfying the following conditions:
  \begin{enumerate}
  \item $[x,\delta_{C}(1_{C})]=0$;
  \item $(\delta_{C} \otimes \id_{A})(x) = (\delta_{C}(1_{C}) \otimes 1_{A})(\id_{C} \otimes \Delta)(x) = (\id_{C} \otimes \Delta)(x)(\delta_{C}(1_{C}) \otimes 1_{A})$;
  \item $x(1_{C} \otimes A)$ and $(1_{C} \otimes A)x$ lie in $C\otimes A$;
  \item $(\id_{C} \otimes \Delta)(x) (1_{C} \otimes 1_{A} \otimes A)$ and $(1_{C}\otimes 1_{A} \otimes A)(\id_{C} \otimes \Delta)(x)$ lie in $M(C\otimes A) \otimes A$.
  \end{enumerate}
 Then $C\boxtimes A$ is a $C^{*}$-algebra,  $\id_{C} \otimes \Delta$ restricts to a coaction of $(A,\Delta)$ on $C\boxtimes A$, and $(C\boxtimes A, \id_{C} \otimes \Delta, \delta_{C})$ is a dilation of $\delta_{C}$.
\end{proposition}
\begin{proof}
  Clearly, $C\boxtimes A$ is a $C^{*}$-algebra.  It contains $\delta_{C}(C)$ by \eqref{eq:partial-coaction} and regularity of $\delta_{C}$.  Next, we need to show that
\begin{align*}
  (\id_{C} \otimes \Delta)(C\boxtimes A) (1_{C} \otimes 1_{A} \otimes A) \subseteq (C \boxtimes A) \otimes A.
\end{align*}
Condition (4) implies that the left hand side is contained in $M(C\otimes A) \otimes A$. Since $A$ has the slice map property, it  suffices to show that for every $y \in C \boxtimes A$ and $\omega \in A^{*}$, the element
\begin{align*}
  x:=(\id_{C} \otimes \id_{A} \otimes \omega)(\id_{C} \otimes \Delta)(y)  = (\id_{C} \otimes (\id_{A} \otimes \omega)\Delta)(y)
\end{align*}
lies in $C \boxtimes A$, that is, satisfies conditions (1)--(4) above. In case of (2)--(4), we only prove the first halfs of the statements, the others follow similarly.

(1) The element $x$ commutes with $\delta_{C}(1_{C})$ because $(\id_{C} \otimes \Delta)(y)$ commutes with $(\delta_{C}(1_{C}) \otimes 1_{A})$ by (2), applied to $y$.

(2) We use (1) for  $y$ and coassociativity of $\Delta$ to see that
\begin{align*}
  (\delta_{C} \otimes \id_{A})(x) &= (\id_{C} \otimes \id_{A} \otimes (\id_{A} \otimes \omega) \Delta)(\delta_{C} \otimes \id_{A})(y) \\
&=   (\id_{C} \otimes \id_{A} \otimes (\id_{A} \otimes \omega) \Delta)((\delta_{C}(1) \otimes 1_{A}) (\id_{C} \otimes \Delta)(y)) \\
&= (\delta_{C}(1_{C}) \otimes 1_{A}) (\id_{C} \otimes (\id_{A} \otimes \id_{A} \otimes \omega)\Delta^{(2)})(y) \\
&= (\delta_{C}(1_{C}) \otimes 1_{A})(\id_{C} \otimes \Delta) (\id_{A} \otimes (\id_{A} \otimes \omega)\Delta)(y) \\ & =  (\delta_{C}(1_{C}) \otimes 1_{A})(\id_{C} \otimes \Delta)(x).
\end{align*}

(3)  Write $\omega=a\upsilon$ with $a\in A$ and $\upsilon \in A^{*}$ using Cohen's factorisation theorem, and let $a' \in A$. Then
\begin{align*}
  x(1_{C} \otimes a') = (\id_{C} \otimes \id_{A} \otimes \upsilon)((\id_{C} \otimes\Delta)(y) (1_{C} \otimes a' \otimes a)).
\end{align*}
We use the relation $A \otimes A =[\Delta(A)(A \otimes A]$ and condition (3) on $y$ and find that
$x(1_{C} \otimes a')$ lies in $C \otimes A$ as desired.

(4)  With $a,a',\upsilon$ as above,  
\begin{align*}
  (\id_{C} \otimes \Delta)(x) \cdot (1_{C} \otimes 1_{A}\otimes a') &= (\id_{C} \otimes \id_{A} \otimes \id_{A} \otimes \upsilon)((\id_{C} \otimes \Delta^{(2)})(y)\cdot (1_{C} \otimes 1_{A} \otimes a' \otimes a)).
\end{align*}
We use the relation $A \otimes A =[\Delta(A)(A \otimes A]$ again  and find that
\begin{align*}
  (\id_{C} \otimes \Delta^{(2)})(y) &\cdot (1_{C} \otimes 1_{A} \otimes a' \otimes a) \\  &\in
  (\id_{C} \otimes \id_{A} \otimes \Delta)((\id_{C} \otimes \Delta)(y)\cdot (1_{C} \otimes 1_{A} \otimes A)) \cdot (1_{C} \otimes 1_{A} \otimes A \otimes A).
\end{align*}
 Condition (4), applied to $y$, implies that the expression above lies in $M(C \otimes A) \otimes A \otimes A$. Slicing  the last factor with $\upsilon$, we get $  (\id_{C} \otimes \Delta)(x) \cdot (1_{C} \otimes 1_{A}\otimes a') \in M(C \otimes A) \otimes A$.
\end{proof}
\begin{example}[Case of a partial group action] \label{example:globalization-group}
  Consider the partial coaction $\delta_{C}$ associated to a disconnected partial action $((p_{g}),(\theta_{g})_{g})$ of a discrete group $\Gamma$ on a $C^{*}$-algebra $C$.  Identify $M(C \otimes C_{0}(\Gamma))$  with $C_{b}(\Gamma; M(C))$ and let $f\in C_{b}(\Gamma;M(C))$. Then  conditions (1) and (4) in  Proposition \ref{proposition:dilation-canonical} are automatically satisfied by $f$, condition (3) is equivalent to $f \in C_{b}(\Gamma;C)$, and condition (2)  corresponds to the invariance condition
\begin{align*} 
  \theta_{g}(p_{g^{-1}}f(h)) = p_{g}f(gh)  \quad (g,h \in \Gamma).
\end{align*}
In particular, if $C=C_{0}(X)$ for some locally compact, Hausdorff space $X$, then each $p_{g}$ is the characteristic function of some clopen $D_{g} \subseteq X$, each $\theta_{g}$ is the pull-back along some homeomorphism $\alpha_{g^{-1}}\colon D_{g} \to D_{g^{-1}}$, and the invariance condition above translates into
\begin{align*}
  f(x,gh) = f(\alpha_{g^{-1}}(x),h)  \quad (g,h\in \Gamma, x\in D_{g}),
\end{align*}
so that $f$ descends to the quotient space of $X\times \Gamma$ with respect to the equivalence relation given by
 $(x,gh) \sim (\alpha_{g^{-1}}(x),h)$ for all $g,h\in \Gamma$ and $x\in D_{g}$. This space is, up to the reparameterization $(x,g) \mapsto (g^{-1},x)$, the globalization of the partial action $((D_{g})_{g},(\alpha_{g})_{g})$ of $\Gamma$ on $X$, see \cite[Theorem 3.5, Proposition 5.5]{exel:book}, and  $C_{0}(X) \boxtimes C_{0}(\Gamma)$ can be identified with a $C^{*}$-subalgebra of $C_{b}((X\times \Gamma)/_{\sim})$. 
\end{example}

\section{Minimal dilations}
\label{sec:minimal}

Among all dilations of a fixed partial coaction $\delta_{C}$ of a $C^{*}$-bialgebra $(A,\Delta)$, we now single out a universal one, which we call the globalization of $\delta_{C}$. More precisely, we  show that (1) every dilation of $\delta_{C}$ contains one that is minimal in a natural sense, and (2) that all such minimal dilations are isomorphic. We need to assume, however, that $\delta_{C}$ is  regular and injective,  that $A$ has the slice map property, and, for (2), that $(A,\Delta)$ is a $C^{*}$-quantum group. 
\begin{definition}
  Let $\delta_{C}$ be a partial coaction of $(A,\Delta)$ on a $C^{*}$-algebra $C$. We call a dilation $(B,\delta_{B},\iota)$ of $\delta_{C}$ \emph{minimal} if $\iota(C)$ and $A^{*} \triangleright \iota(C)$ generate $B$ as a $C^{*}$-algebra.
\end{definition}
\begin{remark}\label{remark:dilation-ideal}
  Let $(B,\delta_{B},\iota)$ be a minimal dilation of a partial coaction $\delta_{C}$ of $(A,\Delta)$ on some $C^{*}$-algebra $C$.   Then $\iota(C) \subseteq B$  is an ideal because $\iota(C)(A^{*} \triangleright \iota(C)) = \iota(C)  \iota(A^{*} \triangleright C) \subseteq \iota(C)$    by  \eqref{eq:morphism-act}. If, moreover, $\iota$ is strict, then $\iota(C)$ is a direct summand of $B$.
  \end{remark}
\begin{example}
  If, in the situation above, $(A,\Delta)$ is the $C^{*}$-bialgebra of functions on a discrete group $\Gamma$, then the coaction $\delta_{B}$ corresponds to an action $\alpha$ of $\Gamma$ on $B$, and the dilation is minimal if and only if   $\sum_{g\in \Gamma} \alpha_{g}(\iota(C))$ generates $B$ as a $C^{*}$-algebra.  
\end{example}
\new \begin{example}[Partial Bernoulli shift]
Let $\qG=(C_{0}(\qG),\Delta)$   be a discrete $C^{*}$-quantum group.  Denote by $\delta$ the coaction of $C_{0}(\qG)$ on $C(\qB_{\qG})$, see Section \ref{section:bernoulli},  by $\delta^{\varepsilon}$ its restriction to a partial coaction on $C(\qB^{\varepsilon}_{\qG})$,  by $\delta^{\times}$ the coaction of $C_{0}(\qG)$ on $C_{0}(\qB^{\times}_{\qG})$ obtained as the restriction of  $\delta$, see Proposition \ref{proposition:bernoulli-coaction}, and by $\iota \colon C(\qB^{\varepsilon}_{\qG}) \hookrightarrow C_{0}(\qB^{\times}_{\qG})$ the inclusion. Then $(C_{0}(\qB^{\times}_{\qG}),\delta^{\times},\iota)$ is a  dilation of $\delta^{\varepsilon}$ because $\delta^{\varepsilon}$ is a restriction of $\delta^{\times}$, and this dilation is minimal. Indeed, $\delta^{\times}(p_{\varepsilon}) = p$ by Lemma \ref{lemma:bernoulli-counit}, whence $C_{0}(\qG)^{*} \triangleright \iota(C_{0}(\qB^{\varepsilon}_{\qG}))$ contains   $p^{\alpha}_{ij}$ for every $\alpha,i,j$, and these elements generate $C_{0}(\qB^{\times}_{\qG})$.
\end{example} \old
Every dilation contains a minimal one:
\begin{proposition} \label{proposition:dilation-minimal} Let $\delta_{C}$ be a  partial coaction of $(A,\Delta)$ on a $C^{*}$-algebra $C$ with a dilation $(B,\delta_{B},\iota)$, and suppose that $A$ has the slice map property.  Denote by $B_{0}\subseteq B$ the $C^{*}$-subalgebra generated by $\iota(C)$ and $A^{*} \triangleright \iota(C)$. 
  \begin{enumerate}
  \item  $\delta_{B}$ restricts to a coaction $\delta_{B_{0}}$ on
    $B_{0}$, and $(B_{0},\delta_{B_{0}},\iota)$ is a minimal dilation
    of $\delta_{C}$.
  \item If $(A,\Delta)$ is a regular $C^{*}$-quantum group and $\delta_{C}$ is weakly continuous,  then $[A^{*} \triangleright \iota(C)] \subseteq B$  is  a $C^{*}$-algebra. If additionally $\iota$ is strict, then  $B_{0}=[(A^{*} \triangleright \iota(C))(\C 1_{B} + \C \iota(1_{C}))]$.  
  \end{enumerate}
\end{proposition}
\begin{proof}
(1) To prove the first assertion, we only need to show that
\begin{align*}
  \delta_{B}(\iota(C)) (1_{B} \otimes A) \subseteq B_{0} \otimes A \quad \text{and} \quad
  \delta_{B}(A^{*} \triangleright \iota(C)) (1_{B} \otimes A) \subseteq  B_{0} \otimes A.
\end{align*}
But for all $c\in C$, $\upsilon,\omega \in A^{*}$,  both
$(\id \otimes \omega)(\delta_{B}(\iota(c))) = \omega \triangleright \iota(c)$ and
$(\id \otimes \omega)(\delta_{B}(\upsilon\triangleright \iota(c))) = \omega\upsilon \triangleright \iota(c)$
lie in $B_{0}$. Since $A$ has the slice map property, the desired inclusions follow.

(2) We  follow the proof of \cite[Proposition 5.7]{baaj:mu-semi-regular}, 
 using the same notation  and manipulations as in the proof of Proposition \ref{proposition:regular-weak-strong}. To shorten the notation,  let $U:=(\pi \otimes \id_{\hat\pi(\hat{A})})(V)$ and $\delta_{\pi}:=(\id_{B} \otimes \pi)\circ \delta_{B} \circ \iota$. Then by \eqref{eq:delta-act},
    \begin{align*}
       [A^{*} \triangleright \iota(C)] &= [A^{*} \triangleright \iota(C(A^{*} \triangleright C))] \\
&=[(\id_{B} \otimes \upsilon\otimes \omega)((\delta_{B} \otimes \id_{A})((C \otimes 1_{A})\delta_{C}(C))) : \upsilon,\omega \in A^{*}] \\
&=[(\id_{B} \otimes \upsilon\otimes \omega)((\delta_{B} \otimes \id_{A})((C \otimes 1_{A})\delta_{B}(C))) : \upsilon,\omega \in A^{*}] \\
&= [(\id_{B} \otimes \upsilon\circ \pi \otimes \omega\circ \pi)((\delta_{B}(C) \otimes 1_{A})(\id_{B} \otimes \Delta)\delta_{B}(C))  : \upsilon,\omega \in \mathcal{B}(K)_{*}] \\
&= [(\id_{B} \otimes \upsilon \otimes \omega)(\delta_{\pi}(C)_{12}U_{23}  \delta_{\pi}(C)_{13}U_{23}^{*})  : \upsilon,\omega \in \mathcal{B}(K)_{*}]\\
&= [(\id_{B} \otimes \upsilon \otimes \omega)(\delta_{\pi}(C)_{12}U_{23} \delta_{\pi}(C)_{13}(\pi(A) \otimes \hat\pi(\hat{A}))_{23})  : \upsilon,\omega \in \mathcal{B}(K)_{*}] \\
&= [(\id_{B} \otimes \upsilon \otimes \omega)(\delta_{\pi}(C)_{12}(\pi(A) \otimes \hat\pi(\hat{A}))_{23} \delta_{\pi}(C)_{13})  : \upsilon,\omega \in \mathcal{B}(K)_{*}]  \\
&= [(A^{*} \triangleright \iota(C))(A^{*} \triangleright \iota(C))].  
    \end{align*}
Thus, $[A^{*} \triangleright \iota(C)]$ is a $C^{*}$-algebra.  If $\iota$ is strict so that $\iota(1_{C})$ is well-defined, then this $C^{*}$-algebra commutes with $\iota(1_{C})$, and by \eqref{eq:morphism-act} the product is $[\iota(A^{*} \triangleright C)] = \iota(C)$. This proves the last assertion concerning $B_{0}$.
\end{proof}

If we apply Proposition \ref{proposition:dilation-minimal} to the canonical dilation  $(C\boxtimes A, \id_{C} \otimes \Delta, \delta_{C})$ constructed 
in  Proposition \ref{proposition:dilation-canonical}, we obtain the following dilation:
\begin{theorem} \label{theorem:globalization}
Let $(A,\Delta)$ be a $C^{*}$-bialgebra, where $A$ has the slice map property, and let $\delta_{C}$ be an injective, regular partial coaction of $(A,\Delta)$ on a $C^{*}$-algebra $C$. Denote by $\Glob(C) \subseteq M(C \otimes A)$  the $C^{*}$-subalgebra generated by 
\begin{align*}
  \{ (\id_{C}  \otimes \id_{C} \otimes  \omega)(\id_{C} \otimes \Delta)\delta_{C}(c) :\omega \in A^{*}, c\in C\} \quad \text{and} \quad \delta_{C}(C).
\end{align*}
Then $\id_{C} \otimes \Delta$ restricts to a partial coaction on $\Glob(C)$ and
\begin{align*}
  \Glob(\delta_{C}):=(\Glob(C),\id_{C} \otimes \Delta,\delta_{C})
\end{align*}
 is a minimal dilation of $\delta_{C}$. 
\end{theorem}
\begin{proof}
By a similar argument as in the proof of Proposition \ref{proposition:dilation-canonical},
  we only need to show that for every $c\in C$ and $\upsilon,\omega \in A^{*}$, the elements
  \begin{align*}
(\id_{C} \otimes \id_{A} \otimes \upsilon) ( (\id_{C} \otimes \Delta)\delta_{C}(c))
\end{align*}
and
\begin{align*}
(\id_{C} \otimes \id_{A} \otimes \upsilon)(    (\id_{C} \otimes \Delta)((\id_{C} \otimes \id_{A} \otimes \omega)(\id_{C} \otimes \Delta)\delta_{C}(c)))
  \end{align*}
lie in $\Glob(C)$. In  the first case, this is trivially true, and in the second case, one finds that the element is equal to
$d = ( \id_{C} \otimes \id_{A} \otimes \upsilon\omega) ((\id_{C} \otimes \Delta)\delta_{C}(C)) \in \Glob(C)$. 
\end{proof}
\begin{remark} Suppose that $(A,\Delta)$ and $\delta_{C}$ are as above.
  \begin{enumerate}
  \item Beware that $\delta_{C}$ is strict as a map from $C$ to $M(C\otimes A)$, but this does not imply that $\delta_{C}$ is strict as a map from $C$ to $\Glob(C)$.
  \item If $\delta_{C}$ is weakly continuous, then \eqref{eq:delta-act} implies that $\Glob(C) \subseteq M(C\otimes A)$ is equal to the 
$C^{*}$-subalgebra generated by
$  \{ (\id_{C}  \otimes \id_{C} \otimes  \omega)(\id_{C} \otimes \Delta)\delta_{C}(c) :\omega \in A^{*}, c\in C\}$ and $\delta_{C}(1_{C})$.
  \end{enumerate}
\end{remark}
\begin{example}  [Case of a partial group action] \label{example:globalization-group-2} Consider the partial coaction $\delta_{C}$ associated to a disconnected partial action $((p_{g}),(\theta_{g})_{g})$ of a discrete group $\Gamma$ on a $C^{*}$-algebra $C$, and identify $M(C\otimes C_{0}(\Gamma))$ with $C_{b}(\Gamma;M(C))$.  In that case, $\Glob(C)$ is the $C^{*}$-algebra generated by all functions of the form
  \begin{align*}
    f_{c,h} = (\id_{C} \otimes \id_{C_{0}(\Gamma)} \otimes \ev_{h})(\id_{C} \otimes \Delta)\delta_{C}(c) \colon  g\mapsto  \theta_{gh}(p_{h^{-1}g^{-1}}c),
  \end{align*}
where $c\in C$ and $g,h\in \Gamma$. The action $\rho$ of $\Gamma$ corresponding to the coaction $\id_{C} \otimes \Delta$ is given by right translation of functions, whence $\rho_{h'} (f_{c,h}) = f_{c,h'h}$ for all $h'\in \Gamma$.
\end{example}
We shall use the following notion of a morphism between dilations:
\begin{definition}
Let  $\delta_{C}$ be a partial coaction of a $C^{*}$-bialgebra $(A,\Delta)$ on some $C^{*}$-algebra $C$. A morphism between dilations $\mathcal{B}=(B,\delta_{B},\iota^{B})$ and $\mathcal{D}=(D,\delta_{D},\iota^{D})$ of $\delta_{C}$ is a $*$-homomorphism $\phi \colon B \to D$ satisfying
\begin{align} \label{eq:dilation-morphism}
  \phi(\iota^{B}(c)) = \iota^{D}(c) \quad \text{and} \quad \delta_{D}(\phi(b))(1_{D} \otimes a) = (\phi \otimes \id_{A})(\delta_{B}(b)(1_{B}\otimes a))
\end{align}
for all $c\in C$, $b\in B$ and $a\in A$. Evidently, all dilations of a fixed partial coaction $\delta_{C}$ form a category; we denote this category by $\Dil(\delta_{C})$.
\end{definition}
\begin{remark}
  The second equation in \eqref{eq:dilation-morphism} is equivalent to the condition that $\phi$ is a morphism of left $A^{*}$-modules, that is, $\omega \triangleright \phi(b) =\phi(\omega \triangleright b)$ for all $b\in B$ and $\omega \in A^{*}$.
\end{remark}

 If $\delta_{C}$ is injective and regular, then the dilation $\Glob(\delta_{C})$ is  terminal among the minimal ones:
 \begin{proposition} \label{proposition:dilation-canonical-morphism} Let $\delta_{C}$ be an injective, regular partial coaction of a $C^{*}$-bialgebra $(A,\Delta)$ on a $C^{*}$-algebra $C$, let $\mathcal{B}=(B,\delta_{B},\iota)$ be a minimal dilation of $\delta_{C}$, and suppose that $A$ has the slice map property.   Then there exists a unique morphism $\phi_{\mathcal{B}}$ from $\mathcal{B}$ to $\Glob(\delta_{C})$, and on the level of $C^{*}$-algebras, $\phi_{\mathcal{B}}$ is surjective.
  For each $b\in B$, the image $\phi_{\mathcal{B}}(b)$ is the restriction of $\delta_{B}(b)$ to the ideal $\iota(C)\otimes A \cong C\otimes A$ in $B\otimes A$.
\end{proposition}
\begin{proof}
 Uniqueness follows from the fact that $B$ is generated by $\iota(C)$ and $A^{*} \triangleright \iota(C)$.

To prove existence,  define $\phi_{\mathcal{B}}$  as in (3).  Since $\iota$ is a weak morphism, $\phi_{\mathcal{B}} \circ \iota =\delta_{C}$. The relation  $(\id_{B} \otimes \Delta)\delta_{B} = (\delta_{B} \otimes \id_{A})\delta_{B}$ implies that 
 \begin{align*}
   (\phi_{\mathcal{B}} \otimes \id_{A})(\delta_{B}(b)(b' \otimes a)) =  (\id_{C} \otimes \Delta)(\phi_{\mathcal{B}}(b))(\phi_{\mathcal{B}}(b') \otimes a)
 \end{align*}
for all $b,b'\in B$ and $a\in A$; in particular,
\begin{align*}
\phi_{\mathcal{B}}  (\omega \triangleright \iota(c)) \phi_{\mathcal{B}}(b) = (\id_{C} \otimes \id_{A} \otimes \omega)((\id_{C} \otimes \Delta)\delta_{C}(C))  \phi_{\mathcal{B}}(b)
\end{align*}
for all $c \in C$ and $\omega \in C^{*}$. Now, the definition of $\Glob(C)$ and minimality of $\mathcal{B}$ imply $\phi_{\mathcal{B}}(B)=\Glob(C)$.\end{proof}
If $(A,\Delta)$ is a $C^{*}$-quantum group, then the  morphism   above is injective and hence an isomorphism. To show this, we use the following observation:
\begin{lemma} 
  Let $\delta_{B}$ be a coaction of a $C^{*}$-quantum group $(A,\Delta)$ on a $C^{*}$-algebra $B$ and let $b,b' \in M(B)$. Then $\delta_{B}(b)(b'\otimes 1_{A}) = 0$ if and only if $(b\otimes 1_{A})\delta_{B}(b') = 0$.
\end{lemma}
\begin{proof}
 Choose a modular  multiplicative  unitary $W$ for $(A,\Delta)$ so that $\Delta(a)=W(a\otimes 1)W^{*}$ for all $a\in A$. Then
  \begin{align*}
  (\delta_{B}\otimes \id_{A})(\delta_{B}(b)) \cdot (\delta_{B}(b')\otimes 1_{A}) &=   (\id_{B} \otimes \Delta)(\delta_{B}(b)) \cdot (\delta_{B}(b') \otimes 1_{A}) \\ &= W_{23}(\delta_{B}(b) \otimes 1_{A})W_{23}^{*} (\delta_{B}(b') \otimes 1_{A}).
  \end{align*}
  Since $\delta_{B} \otimes \id_{A}$ is injective and $W$ is unitary, we can conclude   that  $ \delta_{B}(b)(b'\otimes 1_{A})=0$ if
  \begin{align}
    \label{eq:1}
    (\delta_{B}(b) \otimes 1_{A})  W_{23}^{*} (\delta_{B}(b') \otimes 1_{A}) = 0.
  \end{align}
 A similar  argument shows that  $(b\otimes 1_{A})\delta_{B}(b') = 0$ if and only if
  \begin{align}
    \label{eq:2}
    (\delta_{B}(b) \otimes 1_{A})  W_{23} (\delta_{B}(b') \otimes 1_{A}) = 0.
  \end{align}
Now, both \eqref{eq:1} and \eqref{eq:2} are equivalent to the condition $ \delta_{B}(b) (1_{B} \otimes \hat{A}) \delta_{B}(b')=0$.
  \end{proof}

We can now prove claim (2) stated in the introduction to this section:
\begin{proposition} \label{proposition:dilation-injective}
  Let $\delta_{C}$ be an injective, regular partial coaction of a  $C^{*}$-quantum group $(A,\Delta)$ on a $C^{*}$-algebra $C$,  suppose that $A$ has the slice map property, and let $\mathcal{B}$ be a minimal dilation of $\delta_{C}$. Then the morphism $\phi_{\mathcal{B}}$ from $\mathcal{B}$ to $\Glob(\delta_{C})$ is an isomorphism.
\end{proposition}
\begin{proof}
Write $\mathcal{B}=(B,\delta_{B},\iota)$. It suffices to show that $\phi_{\mathcal{B}}$ is injective on the level of $C^{*}$-algebras. 
On the direct summand $\iota_{C}(C) \subseteq B$, the morphism $\phi_{\mathcal{B}}$ is given by $\iota_{C}(c)\mapsto \delta_{C}(c)$ and hence injective.
Since $B$ is minimal, the direct summand $(1_{B}-\iota(1_{C}))B$ of $B$ is generated by $(1_{B} - \iota(1_{C}))(A^{*} \triangleright \iota(C))$. Given a non-zero $b \in (1_{B}-\iota(1_{C}))B$, we therefore find some $c\in C$ such that $\delta_{B}(\iota(c))(b\otimes 1_{A})$ is non-zero, and then $(\iota(c) \otimes 1_{A})\delta_{B}(b)$ is non-zero by the lemma above, whence $\phi_{\mathcal{B}}(b)$ is non-zero.
\end{proof}
\begin{corollary} \label{corollary:dilation-minimal}
Let $(A,\Delta)$ be a $C^{*}$-quantum group and suppose that $A$ has the slice map property. Then all minimal dilations of an injective, regular partial coaction of $(A,\Delta)$ are isomorphic.
\end{corollary}

\section*{Acknowledgements} 
We  thank the referees for helpful comments that improved the quality of the manuscript. 
\bibliographystyle{amsplain}

\def\cprime{$'$}
\providecommand{\bysame}{\leavevmode\hbox to3em{\hrulefill}\thinspace}
\renewcommand{\MR}[1]{}
\providecommand{\MRhref}[2]{%
}
\providecommand{\href}[2]{#2}

\end{document}